\documentclass[letterpaper, 10 pt, conference]{ieeeconf}  % Comment this line out
\IEEEoverridecommandlockouts                              % This command is only
\usepackage{graphicx}
\usepackage{graphics} % for pdf, bitmapped graphics files
\usepackage{amsmath} % assumes amsmath package installed
\usepackage{amssymb}  % assumes amsmath package installed
\usepackage{amsfonts}
\usepackage{blkarray}
\usepackage{arydshln}
\usepackage{algorithmic}
\usepackage{algorithm}
\usepackage{url}
\usepackage{color}
\usepackage[compress]{cite}

\newtheorem{theorem}{\bf Theorem}

\newtheorem{problem}{\bf Problem}

\newtheorem{definition}{\bf Definition}

\newtheorem{proposition}{\bf Proposition}
\newtheorem{lemma}{\bf Lemma}
\newtheorem{assumption}{\bf Assumption}

\def\QED{~\rule[-1pt]{5pt}{5pt}\par\medskip}

\newcommand*{\LONGVERSION}{}% 

\title{\LARGE \bf
Optimal Output Feedback Architecture for Triangular LQG Problems
}

\author{Takashi Tanaka and Pablo A. Parrilo % <-this % stops a space
\thanks{Both authors are with Laboratory for Information and Decision Systems,
        Massachusetts Institute of Technology, Cambridge, MA, USA
        {\tt\small ttanaka@mit.edu, parrilo@mit.edu}}%
}

\begin{document}

\maketitle
\thispagestyle{empty}
\pagestyle{empty}

%%%%%%%%%%%%%%%%%%%%%%%%%%%%%%%%%%%%%%%%%%%%%%%%%%%%%%%%%%%%%%%%%%%%%%%%%%%%%%%%
\begin{abstract}
Distributed control problems under some specific information constraints can be formulated as (possibly infinite dimensional) convex optimization problems. The underlying motivation of this work is to develop an understanding of the optimal decision making architecture for such problems.
In this paper, we particularly focus on the $N$-player triangular LQG problems and show that the optimal output feedback controllers have attractive state space realizations. The optimal controller can be synthesized using a set of stabilizing solutions to $2N$ linearly coupled algebraic Riccati equations, which turn out to be easily solvable under reasonable assumptions.
% We consider a cascaded interconnection of $N$ linear time invariant dynamical systems in which system dynamics propagates only in one direction on the chain. 
% The objective of this paper is to synthesize the optimal distributed controller in $H_2$ sense for such systems under information constraint that the local control input depends only on the outputs of upstream subsystems on the chain.
% This problem has the quadratic invariance property, and its tractability has been well studied in the literature.
% In this paper, we particularly show that the optimal controller is constructed using a set of stabilizing solutions to $2N$ linearly coupled algebraic Riccati equations, which can be easily solved under reasonable assumptions. This work extends the recent result for the case of $N=2$. 
\end{abstract}

%%%%%%%%%%%%%%%%%%%%%%%%%%%%%%%%%%%%%%%%%%%%%%%%%%%%%%%%%%%%%%%%%%%%%%%%%%%%%%%%
\section*{Nomenclature}

\begin{list}{\labelitemi}{\leftmargin=0em}
\item 
$\mathbf{G}(s)=\left[ \begin{array}{c|c}A & B \\ \hline 
 C & D \end{array}\right]:=C(sI-A)^{-1}B+D$ represents a proper rational transfer function matrix. 
 
\item The space of matrix valued functions $\mathbf{G}$ such that $\|\mathbf{G}\|\triangleq \sqrt{\left<\mathbf{G},\mathbf{G}\right>}<\infty$ is denoted by $L_2$, where $\left<\mathbf{G}_1, \mathbf{G}_2\right>\triangleq \frac{1}{2\pi}\int_{-\infty}^{\infty} tr \mathbf{G}_1^*(j\omega)\mathbf{G}_2(j\omega) d\omega$.
$L_2$ can be written as $L_2=H_2\oplus H_2^\perp$, where $H_2$ is the subspace of functions $\mathbf{G}$ analytic in $Re(s)>0$. The $H_2$ norm of a function $\mathbf{G}\in H_2$ can be computed using the same formula $\|\mathbf{G}\|\triangleq \sqrt{\left<\mathbf{G},\mathbf{G}\right>}$. The Hardy space $H_\infty$ will be also used in this paper.

\item Let $I_n$ be the $n=n_1+\cdots+n_N$ dimensional identity matrix. Denote by $E^i$ the $i$-th block column of $I_n$, and by $E_i$ the $i$-th block row of $I_n$. We also write $E^{\uparrow  i}$ to denote the first $n_1+\cdots+n_i$ columns of $I_n$, while $E^{\downarrow i}$ is the last $n_i+\cdots+n_N$ columns of $I_n$. Similarly, $E_{\uparrow  i}$ denotes the first $n_1+\cdots+n_i$ rows of $I_n$, while $E_{\downarrow i}$ is the last $n_i+\cdots+n_N$ rows of $I_n$. 
For a general matrix $M$, shorthand notations such as $M_{\uparrow i}=E_{\uparrow i}M$, $M^{\uparrow i}=ME^{\uparrow i}$, $M_{\uparrow i}^{\uparrow j}=E_{\uparrow i}ME^{\uparrow j}$ will be also used.

\item $\mathcal{I}$ specifies a sparsity structure of matrices, or matrix-valued functions. If $M$ has $N\times N$ subblocks, $M\in\mathcal{I}_{LBT}$ means that $M$ is lower block triangular. If $M\in \mathcal{I}_{\downarrow i}^{\uparrow i}$, then $M$ has nonzero components only in the $(\cdot)_{\downarrow i}^{\uparrow i}$ subblock. If $M\in\mathcal{I}_i^j$, only $(i,j)$-th subblock can be nonzero.

\item $(X,K)=ARE_p(A,B,F,H)$ represents a solution of algebraic Riccati equation
$
A^TX+XA-(XB+F^TH)\Psi^{-1}(XB+F^TH)^T+F^TF=0
$
with $K=-\Psi^{-1}(XB+F^TH)^T$ where $\Psi\triangleq H^TH$. Similarly, $(Y,L)=ARE_d(A,C,W,V)$ represents a solution of
$
AY+YA^T-(CY+VW^T)^T\Phi^{-1}(CY+VW^T)+WW^T=0
$
with $L=-(CY+VW^T)^T\Phi^{-1}$ where $\Phi=VV^T$.

\item $\text{row}\{M_1,\cdots,M_n\}\triangleq [M_1 \; \cdots \; M_n]$, $\text{col}\{M_1,\cdots,M_n\}\triangleq [M_1^T \; \cdots \; M_n^T]^T$.
\end{list}

%%%%%%%%%%%%%%%%%%%%%%%%%%%%%%%%%%%%%%%%%%%%%%%%%%%%%%%%%%%%%%%%%%%%%%%%%%%%%%%%
\section{Introduction}

It is widely recognized that tractability of distributed control problems greatly depends on the information structure underlying the problem. If the information structure is arbitrary, the problem can be hopelessly hard as demonstrated by an iconic example by Witsenhausen in 1968. 
In contrast, many tractability results initiated by Ho and Chu \cite{RefWorks:271} suggest that distributed control problems seem much more accessible when decision makers form a hierarchy in terms of their ability to observe and control the physical system.
Currently, a unification via the \emph{quadratic invariance (QI)} introduced by \cite{RefWorks:17} is known to capture a wide class of distributed control problems that can be formulated as (infinite dimensional) convex optimization problems. 
Unfortunately, the QI framework does not immediately lead us to an explicit form of the optimal solution, and as a result, state space realizations of the optimal controllers remain unknown for many QI optimal control problems.
This paper derives a state space realization of the solution to the triangular LQG problem, which is a special case but an important instance of the QI optimal control problems.

The triangular LQG problem is formulated as follows.
Suppose that the transfer function of the system to be controlled is given by
$$
\mathbf{G}=\left[ \begin{array}{cc} \mathbf{G}_{11} & \mathbf{G}_{12} \\
\mathbf{G}_{21} & \mathbf{G}_{22} \end{array}\right]\triangleq 
\left[ \begin{array}{c|cc} A&W&B \\ \hline 
F&0&H \\
C&V&0
 \end{array}\right].
$$
Matrices $A\in\mathbb{R}^{n\times n}, B\in\mathbb{R}^{n\times m}, C\in\mathbb{R}^{p\times n}$ are partitioned according to $n=n_1+\cdots+n_N, m=m_1+\cdots+m_N, p=p_1+\cdots+p_N$, and $A,B,C \in \mathcal{I}_{LBT}$ with respect to this partitioning. 
The injected noise $w$, performance output $z$, control input $u$, and observation output $y$ are related by $\text{col}\{z,y\}=\mathbf{G}\text{col}\{w,u\}$.
A controller transfer function $\mathbf{K}\in \mathcal{I}_{LBT}$ needs to be designed so that $u=\mathbf{K}y$ minimizes the $H_2$ norm of the closed loop transfer function from $w$ to $z$.
\begin{problem}
\label{prob1}
Find a state space realization of the optimal solution $\mathbf{K}_{opt}$ to the problem:
\begin{subequations}
\label{h2controlproblem}
\begin{align}
\min \; &\|\mathbf{G}_{11}+\mathbf{G}_{12}\mathbf{K}(I-\mathbf{K}\mathbf{G}_{22})^{-1}\mathbf{G}_{21}\|^2 \\
\text{s.t. } & \mathbf{K}\in \mathcal{I}_{LBT} \text{ and stabilizing.} 
\end{align}
\end{subequations}
\end{problem}
Problem \ref{h2controlproblem} can be interpreted as a distributed control problem under a particular information constraint shown in Fig. \ref{figposet}.
We make some natural assumptions on system matrices $A,B,C,F,H,W$ and $V$ so that Problem \ref{h2controlproblem} is well-posed.
\begin{assumption}
\label{asmp1}
\begin{itemize}
\item[ 1.] For every $i\in \{1,2,\cdots, N\}$, $(A_{ii},B_{ii})$ is stabilizable and $H$ has full column rank.
\item[ 2.] $\left[ \begin{array}{cc} A-j\omega I & B \\ F & H \end{array}\right]$ has full column rank for all $\omega \in \mathbb{R}$.
\item[ 3.] For every $i\in \{1,2,\cdots, N\}$, $(C_{ii},A_{ii})$ is detectable and $V$ has full row rank.
\item[ 4.] $\left[ \begin{array}{cc} A-j\omega I & W \\ C & V \end{array}\right]$ has full row rank for all $\omega \in \mathbb{R}$.
\end{itemize}
\end{assumption}
% Conditions 1 and 3 are necessary for the existence of a structured stabilizing controller. Conditions 2 and 4 guarantee that there are no pole/zero cancellations on the imaginary axis in the closed loop system.
Due to the \emph{quadratic invariance} property \cite{RefWorks:17}, Problem \ref{h2controlproblem} can be written as a convex optimization problem by introducing a particular re-parametrization \cite{RefWorks:277}. It is also straightforward to see by \emph{vectorization} \cite{RefWorks:293} that Problem \ref{h2controlproblem} admits a unique and rational solution in this parameter domain, and hence that $\mathbf{K}_{opt}$ is also rational.
In this paper, we further show that the optimal controller has a fascinating state space structure, which can be easily synthesized by solving a set of linearly coupled Riccati equations. This work extends recent progress in the understanding of state space solutions to distributed control problems \cite{RefWorks:265,RefWorks:269,RefWorks:294,RefWorks:273,RefWorks:267}.

\ifdefined\DEBUG
\else
 \begin{figure}[tb]
 \begin{center}
 \includegraphics[width=0.6\linewidth, bb=100 80 600 450]{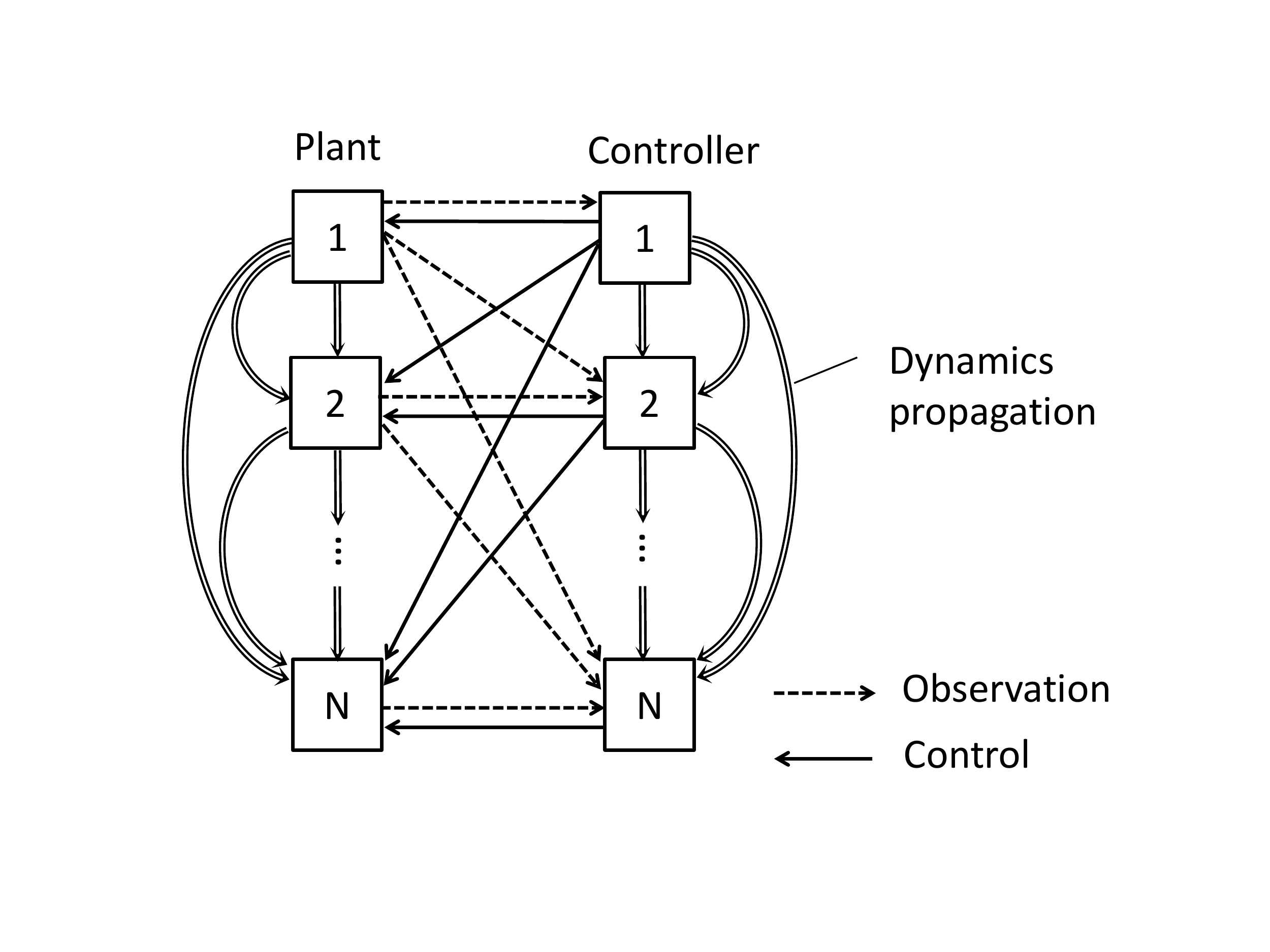}
 \caption{$A,B,C \in \mathcal{I}_{LBT}$ means that system dynamics propagates only downward on the chain of local subsystems (defined by $A_{ii},B_{ii},C_{ii}$).
Each row of $\mathbf{K}\in \mathcal{I}_{LBT}$ can be considered as an independent controller, who controls local subsystem based on the observations of the outputs of upstream subsystems. 
Alternatively, each column of $\mathbf{K}$ can be seen as an independent controller, who observes local output and controls downstream subsystems. Another valid interpretation is to view $\mathbf{K}\in \mathcal{I}_{LBT}$ as a sum of local controllers $\mathbf{K}_i\in \mathcal{I}_{\downarrow i}^{\uparrow i}$, who observes upstream subsystems and controls downstream systems (the above figure).}
 \label{figposet}
 \end{center}
 \end{figure}
\fi

\section{Summary of the result}
Under Assumption \ref{asmp1} and \ref{asmp2} (the second assumption will be discussed later),  a state space model of $\mathbf{K}_{opt}$ can be constructed.
This requires to determine $N$ controller gains $K_1,\cdots,K_N$ and $N$ observer gains $L_1,\cdots,L_N$ by finding a set of stabilizing solutions to algebraic Riccati equations: 
\begin{subequations}
\label{coupledrc}
\begin{align}
(X_1, K_1)&=ARE_p(A,B,F,H) \label{rc1c}\\
(X_i, K_i)&=ARE_p(A\!+\!L_{i-1}C_{\uparrow i-1},B^{\downarrow i}, \nonumber \\
 &-\!H^{\downarrow i-1}K_{i-1},H^{\downarrow i}), \;\; i\in\{2,3,\cdots,N \}\label{rcic}\\
(Y_N, L_N)&=ARE_d(A,C,W,V) \label{rcNf}\\
(Y_i,L_i)&=ARE_d(A\!+\!B^{\downarrow i+1}K_{i+1}, C_{\uparrow i}, \nonumber \\
&-\!L_{i+1}V_{\uparrow i+1}, V_{\uparrow i}),\;\; i\in\{1,2,\cdots,N\!-\!1 \}.\label{rcif}
\end{align}
\end{subequations}
Notice that (\ref{rc1c}) and (\ref{rcNf}) can be solved independently since they require problem data $A,B,C,F,H,W$ and $V$ only, while the remaining $2N-2$ Riccati equations (\ref{rcic}) and (\ref{rcif}) have dependencies on each other.
By carefully looking at their substructures, we show that unknown variables can be sequentially determined as shown in Fig. \ref{fig23}. Apparently, concepts of control and estimation are highly symmetric in this synthesis.

It is convenient to introduce $nN$ dimensional square incidence matrices $\zeta$ and $\mu$ to describe the architecture of $\mathbf{K}_{opt}$. Define $\zeta$ by dividing it into $N\times N$ sub-blocks and setting its $(i,j)$-th sub-block to $I_n$ if $i\geq j$ and to zero otherwise. Matrix $\mu$ is defined as the inverse of $\zeta$.

\begin{theorem}
\label{maintheo}
Under Assumption \ref{asmp1} and \ref{asmp2}, (\ref{coupledrc}) admits a unique set of solutions $(X_i,K_i), (Y_i,L_i), i=1,\cdots,N$ such that each of them is a stabilizing solution to the corresponding Riccati equation in (\ref{coupledrc}). Using these solutions, the optimal controller to Problem \ref{h2controlproblem} can be written as 
$
\mathbf{K}_{opt}=
\left[ \begin{array}{c|c} 
A_K & B_K \\ \hline 
C_K & 0 \end{array}\right]
$
where $A_K,B_K,C_K$ are defined by
\begin{subequations}
\label{akbkck}
\begin{align}
A_K=&I \otimes  A+
\text{diag}\{ L_1C_{\uparrow 1},\cdots, L_NC_{\uparrow N}\} \nonumber \\
&+\zeta \text{diag}\{ B^{\downarrow  1}K_1, \cdots, B^{\downarrow  N}K_N \}\mu \\
B_K=&-\text{col}\{ L_1E_{\uparrow 1}, \cdots, L_NE_{\uparrow N} \} \\
C_K=&\text{row}\{ E^{\downarrow 1}K_1, \cdots, E^{\downarrow N}K_N \}\mu.
\end{align}
\end{subequations}
Moreover, the optimal value of Problem \ref{h2controlproblem} can be written as $J_{opt}^2=J_{cnt}^2+J_{dcnt}^2$ where
\begin{align*}
&J_{cnt}^2=tr W^TX_1W+ tr \Psi K_1Y_NK_1^T \\
&J_{dcnt}^2\!=\!\!\sum_{j=1}^{N-1}\!\!tr(HK_1\!\!-\!\!H^{\downarrow j+1}K_{j+1})Y_j(HK_1\!\!-\!\!H^{\downarrow j+1}K_{j+1})^T.
\end{align*}
\end{theorem}
Note that $J_{cnt}$ can be interpreted as the optimal cost when the controller is designed without information constraints (compare with the result of standard $H_2$ control). The price to pay to impose information constraints as in Fig. \ref{figposet} is precisely given by $J_{dcnt}$.
The optimal controller given in Theorem \ref{maintheo} turns out to be a certainty equivalent controller. That is, if $\text{col}\{x^{K_1}(t), \cdots, x^{K_N}(t)\}$ is the state of the optimal controller, then $x^{K_i}(t)$ can be interpreted as the least mean square estimate of $x(t)$ based on the observations of outputs of upstream subsystems (see Appendix \ref{secce}).
%\footnote{Certainty equivalence property of related problems is also observed in \cite{RefWorks:278,RefWorks:264,RefWorks:281}.}.
Nevertheless, as Fig. \ref{fig23} shows, controller and observer gain must be jointly designed when $N\geq 2$.
The well-known \emph{separation principle} holds only in an exceptional circumstance of $N=1$, where controller and observer gains can be designed separately.

 \ifdefined\DEBUG
\else
\begin{figure}[tb]
 \begin{center}
 \includegraphics[width=0.9\linewidth, bb=50 120 650 420]{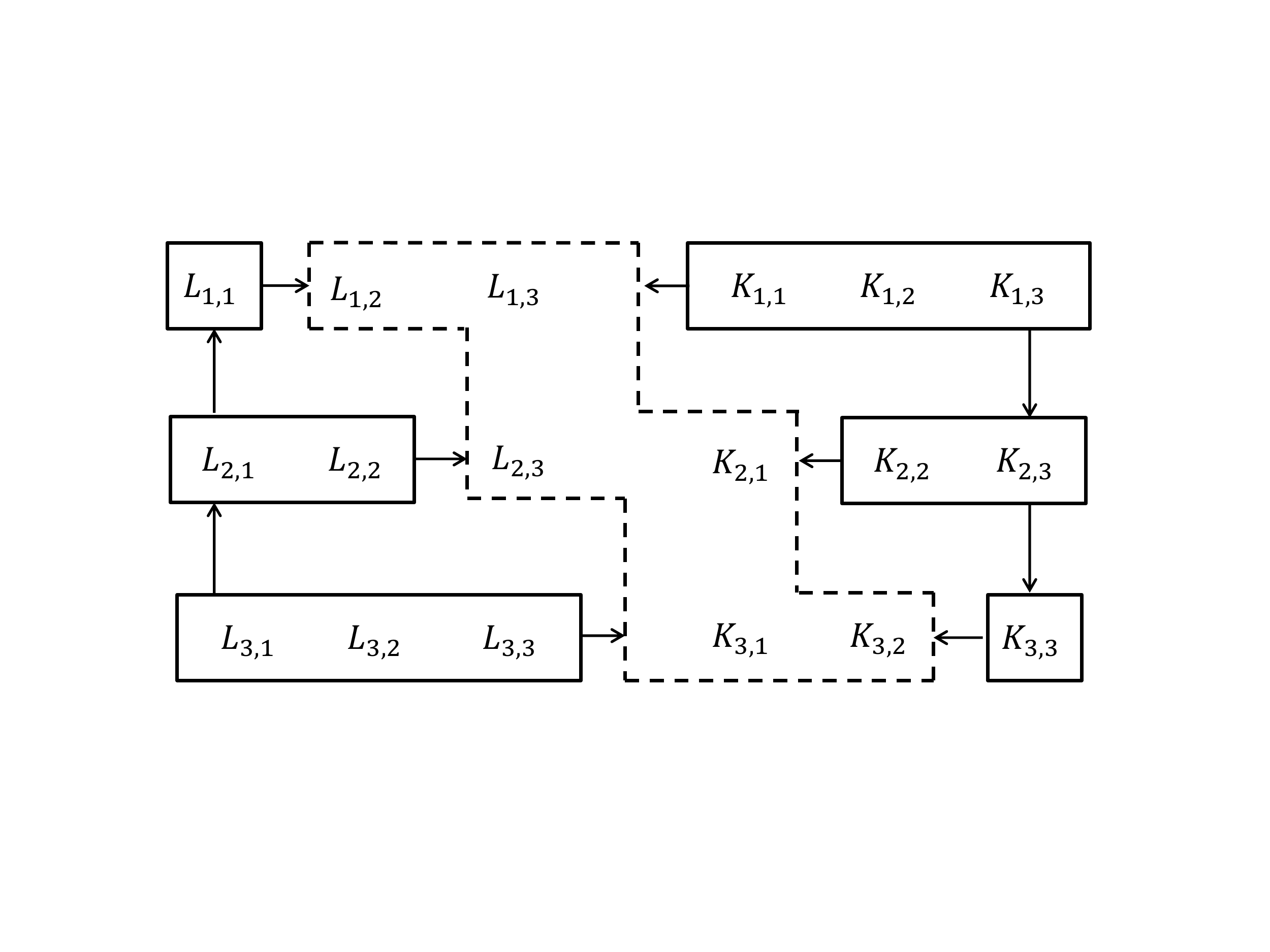}
 \caption{This diagram illustrates the order in which controller/observer gains are determined when $N=3$. Suppose $K_{i,j}$ is the $i$-th player's controller gain acting on his state estimate of $x_j$ (i.e., $K_i=\text{row}\{K_{i,1},K_{i,2},K_{i,3}\}$). Similarly, suppose  $L_{i,j}$ is the $i$-th player's estimator gain to update his state estimate of $x_j$ (i.e., $L_i=\text{col}\{L_{i,1},L_{i,2},L_{i,3}\}$).
The solution process starts by solving (\ref{rc1c}) to generate $K_{1,1},K_{1,2},K_{1,3}$.
This recursively allows one to solve for $K_{2,2},K_{2,3}$ and then for $K_{3,3}$. Estimator gains can be also found by first solving (\ref{rcNf}) and then proceed backward on the chain. Finally, the remaining gains (in the dotted line) is computed at once by solving a system of linear equations. 
This procedure is a natural extension of the case of $N=2$ reported in \cite{RefWorks:294}. }
 \label{fig23}
 \end{center}
 \end{figure}
\fi

\section{Coupled Riccati Equations}
\label{seccoupledrc}

Since the proof of optimality of the proposed controller is closely related to the solvability of the set of Riccati equations (\ref{coupledrc}), we study its solution procedure in this section. 
We introduce the following partitioning of unknown matrices $X_i, Y_i, K_i$ and $L_i$ for every $i\in\{1,2,\cdots,N\}$:
\begin{align}
\label{unknowns}
X_i&=\left[ \begin{array}{cc} \check{X}_i & \bar{X}_i \\
\bar{X}_i^T & \hat{X}_i \end{array}\right], \;\;
K_i=\left[ \begin{array}{cc} \bar{K}_i & \hat{K}_i\end{array}\right], \nonumber \\
Y_i&=\left[ \begin{array}{cc} \hat{Y}_i & \bar{Y}_i \\
\bar{Y}_i^T & \check{Y}_i \end{array}\right], \;\; 
L_i=\left[ \begin{array}{c} \hat{L}_i \\ \bar{L}_i \end{array}\right]
\end{align}
where
$
 \hat{X}_i=(X_i)_{\downarrow i}^{\downarrow i},\;\;
 \hat{Y}_i=(Y_i)_{\uparrow  i}^{\uparrow  i},\;\;
 \hat{K}_i=K_iE^{\downarrow i}, \;\;
 \hat{L}_i=E_{\uparrow i}L_i$.
In particular, $X_1=\hat{X}_1, K_1=\hat{K}_1$ and $Y_N=\hat{Y}_N, L_N=\hat{L}_N$.
Matrices $\hat{K}_i$ and $\hat{L}_i$ are further partitioned as
$\hat{K}_i=\text{row}\{\hat{K}_i^a,\hat{K}_i^b\}, \hat{L}_i=\text{col}\{\hat{L}_i^b,\hat{L}_i^a\}$
where $\hat{K}_i^a=K_iE^i$ and $\hat{L}_i^a=E_iL_i$.
Also introduce
\begin{align*}
% &A_{i,i-1}^{KL}\triangleq \begin{cases}
% A+BK_1 & \mbox{if } i=1 \\
% A+B^{\downarrow i}K_i+L_{i-1}C_{\uparrow i-1} & \mbox{if } i\in\{2,3,\cdots, N\} \\
% A+L_NC  & \mbox{if } i=N+1
% \end{cases} \\
&\hat{A}_i^K\triangleq A_{\downarrow i}^{\downarrow i}+B_{\downarrow i}^{\downarrow i}\hat{K}_i, \;
\bar{A}_i^K\triangleq A_{\downarrow i}^{\uparrow i-1}+B_{\downarrow i}^{\downarrow i}\bar{K}_i \\
&\hat{A}_i^L\triangleq A_{\uparrow  i}^{\uparrow  i}+\hat{L}_{i}C_{\uparrow  i}^{\uparrow  i}, \;
\bar{A}_i^L\triangleq A_{\downarrow  i+1}^{\uparrow i}+\bar{L}_{i}C_{\uparrow  i}^{\uparrow  i} \\
& A_{1,0}^{KL}\triangleq A+BK_1, \; A_{N+1,N}^{KL}\triangleq  A+L_NC \\
&A_{i,i-1}^{KL}\triangleq A+B^{\downarrow i}K_i+L_{i-1}C_{\uparrow i-1} \text{ for } 2\leq i \leq N.
\end{align*}
Equation (\ref{rcic}) can be written as
\begin{subequations}
\label{rcicr}
\begin{align}
{A_{i,i-1}^{KL}}^TX_i+X_iA_{i,i-1}^{KL}+\Sigma_i &=0 \label{rcicr1} \\ 
K_i^T\Psi_{\downarrow i}^{\downarrow i}+X_iB^{\downarrow i}-K_{i-1}^T\Psi_{\downarrow i-1}^{\downarrow i}&=0 \label{rcicr2} 
\end{align}
\end{subequations}
where $\Sigma_i=(H^{\downarrow i}K_i-H^{\downarrow i-1}K_{i-1})^T(H^{\downarrow i}K_i-H^{\downarrow i-1}K_{i-1})$ and (\ref{rcif}) is rearranged as
\begin{subequations}
\label{rcifr}
\begin{align}
A_{i+1,i}^{KL}Y_i+Y_i{A_{i+1,i}^{KL}}^T+\Pi_i &=0 \label{rcifr1} \\ 
\Phi_{\uparrow  i}^{\uparrow  i}L_i^T+C_{\uparrow i}Y_i-\Phi_{\downarrow i}^{\downarrow i+1}L_{i+1}^T&=0 \label{rcifr2} 
\end{align}
\end{subequations}
where $\Pi_i=(L_iV_{\uparrow i}-L_{i+1}V_{\uparrow i+1})(L_iV_{\uparrow i}-L_{i+1}V_{\uparrow i+1})^T$. Now, all unknown variables can be determined by the following three-step procedure, which also visualized in Fig. \ref{figvals}.

\begin{figure}[tb]
 \begin{center}
 \includegraphics[width=0.8\linewidth, bb=0 0 700 500]{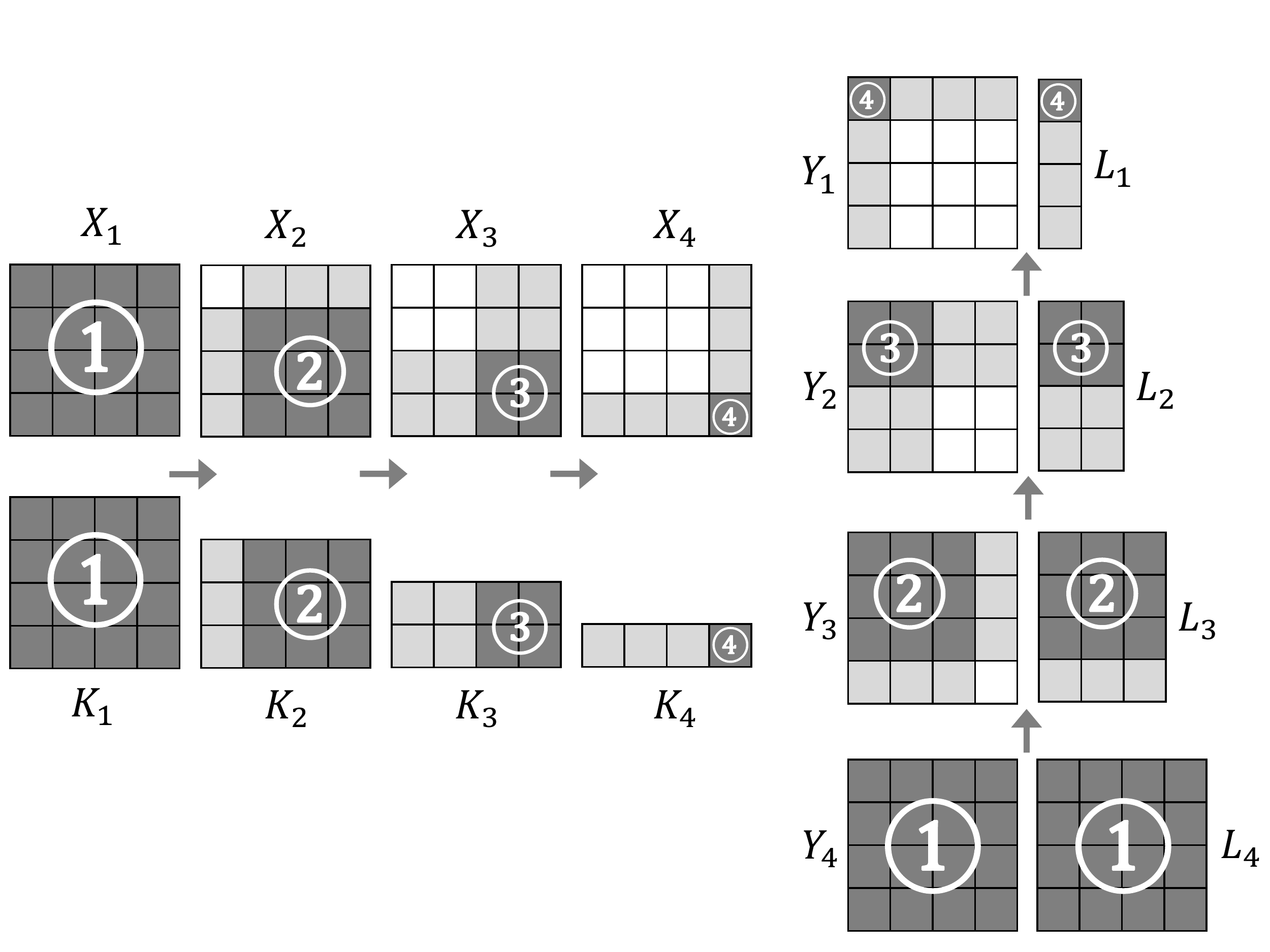}
 \caption{Sequential process to determine unknown variables. Step 1: Entries with dark gray are computed by sequentially solving Riccati subequations. Step 2 : Entries with light gray are computed at once by solving a linear system. Step 3: White entries are obtained by solving Lyapunov equations.}
 \label{figvals}
 \end{center}
 \end{figure}

\subsection{Step 1: Sequential solving of Riccati subequations}
In this step, sub-matrices indicated by `` $\hat{}$ " in (\ref{unknowns}) are determined. Notice that $\hat{X}_1,\hat{K}_1,\hat{Y}_N,\hat{L}_N$ are directly obtained by solving (\ref{rc1c}) and (\ref{rcNf}). 
To compute $\hat{X}_i,\hat{K}_i$ for $i\in\{2,3,\cdots, N\}$, focus on the lower-right $(\cdot)_{\downarrow i}^{\downarrow i}$ sub-block of (\ref{rcicr1}) and $(\cdot)_{\downarrow i}$ subblock of (\ref{rcicr2}). 
They are by themselves Riccati equations with respect to $(\hat{X}_i,\hat{K}_i)$:
\begin{equation}
(\hat{X}_i, \hat{K}_i)=ARE_p(A_{\downarrow i}^{\downarrow i},B_{\downarrow i}^{\downarrow i}, -H^{\downarrow i-1}\hat{K}_{i-1}^b, H^{\downarrow i}).
\label{rcicsub}
\end{equation}
Since the right hand side contains $\hat{K}_{i-1}^b$, these Riccati equations need to be solved in the forwarding (ascending) order on the chain. 
Similarly, $\hat{Y}_i,\hat{L}_i$ for $i\in\{1,2,\cdots,N-1\}$ can be computed by focusing on the upper-left $(\cdot)_{\uparrow  i}^{\uparrow i}$ sub-block of (\ref{rcifr1}) and $(\cdot)^{\uparrow i}$ sub-block of (\ref{rcifr2}). They are Riccati equations with respect to $(\hat{Y}_i,\hat{L}_i)$:
\begin{equation}
(\hat{Y}_i, \hat{L}_i)=ARE_d(A_{\uparrow i}^{\uparrow i},C_{\uparrow i}^{\uparrow i},-\hat{L}_{i+1}^bV_{\uparrow i+1},V_{\uparrow i}).
\label{rcifsub}
\end{equation}
Since the right hand side contains $\hat{L}_{i+1}^b$, they need to be solved in the backward (descending) order in the chain.
\begin{proposition}
\label{propstabsol}
Under Assumption \ref{asmp1}, algebraic Riccati equations (\ref{rcicsub}) and (\ref{rcifsub}) admit a unique positive semidefinite solution, which is also stabilizing.
\end{proposition}

\subsection{Step 2: Solving a linear system}
In this step, we compute components with`` $\bar{}$ ". 
By looking at the upper-right $(\cdot)_{\uparrow i-1}^{\downarrow i}$ subblock of (\ref{rcicr1}) and the upper $(\cdot)_{\uparrow i-1}$ subblock of (\ref{rcicr2}), as well as the upper-right $(\cdot)_{\uparrow i}^{\downarrow  i+1}$ subblock of (\ref{rcifr1}) and the right $(\cdot)^{\downarrow  i+1}$ subblock of (\ref{rcifr2}), we obtain
\begin{subequations}
\label{lineqsys}
\begin{align}
&\bar{K}_i^T \Psi_{\downarrow i}^{\downarrow i}+\bar{X}_i B_{\downarrow i}^{\downarrow i}-\text{row}\{\bar{K}_{i-1}, \hat{K}_{i-1}^a\}^T\Psi_{\downarrow i-1}^{\downarrow i}=0  \label{lina} \\
&{\hat{A}_{i-1}^L\!}^T\bar{X}_i+\bar{X}_i\hat{A}_i^K+{\bar{A}_{i-1}^L\!}^T \hat{X}_i \nonumber \\
&+\text{row}\{\bar{K}_{i-1}, \hat{K}_{i-1}^a\}^T(\Psi_{\downarrow i-1}^{\downarrow i-1}\hat{K}_{i-1}^b-\Psi_{\downarrow i-1}^{\downarrow i}\hat{K}_i)=0  \label{linc} \\
&\!\!\! \text{for every } i \in\{2,3\cdots, N\} \text{ and } \nonumber \\
&\Phi_{\uparrow  i}^{\uparrow  i}\bar{L}_i^T+C_{\uparrow  i}^{\uparrow  i}\bar{Y}_i^T-\Phi_{\uparrow  i}^{\uparrow  i+1}\text{col}\{\hat{L}_{i+1}^a, \bar{L}_{i+1}\}^T=0  \label{linb} \\
&\hat{A}_i^L\bar{Y}_i+\bar{Y}_i{\hat{A}_{i+1}^K \!}^T+\hat{Y}_i{\bar{A}_{i+1}^K\!}^T  \nonumber \\
&+(\hat{L}_{i+1}^b\Phi_{\uparrow i+1}^{\uparrow i+1}-\hat{L}_i\Phi_{\uparrow i}^{\uparrow i+1})\text{col}\{\hat{L}_{i+1}^a, \bar{L}_{i+1}\}^T =0  \label{lind}
\end{align}
\end{subequations}
for every $i \in\{1,2,\cdots, N-1\}$.
Since $\hat{X}_i,\hat{K}_i,\hat{Y}_i,\hat{L}_i$ are computed in the previous step, these are  linear equations with respect to $\bar{X}_i, \bar{K}_i, i\in\{2,3,\cdots,N\}$ and $\bar{Y}_i, \bar{L}_i, i\in\{1,2,\cdots,N-1\}$.
There are precisely the same number of linear constraints as the number of real unknowns.
Unfortunately, we are currently not aware of a theoretical guarantee for the non-singularity of (\ref{lineqsys}). Hence at this point, we have to make an additional assumption:
\begin{assumption}
\label{asmp2}
The linear system (\ref{lineqsys}) with respect to $\hat{X}_i, \hat{K}_i$, $i\in\{2,3,\cdots, N\}$ and $\hat{Y}_i, \hat{L}_i$, $i\in\{1,2,\cdots, N-1\}$ admit a unique solution.
\end{assumption}
% As we will see, if (\ref{lineqsys}) has a solution, the set of algebraic Riccati equation (\ref{coupledrc}) admits a set of stabilizing solutions. 
% By choosing $K_i$ and $L_i$ in this way in (\ref{akbkck}), the assumed form of controller satisfy the necessary and sufficient condition for the optimality.
% However, if (\ref{lineqsys}) fails to admit a solution, there are no such $K_i$ and $L_i$, implying that the optimal controller is not in the assumed form.
When $N=2$, it is shown in \cite{RefWorks:294} that the linear system (\ref{lineqsys}) admits a unique solution under Assumption \ref{asmp1} and thus Assumption \ref{asmp2} is unnecessary. It must be addressed in the future whether this generalizes to $N>2$.
Our numerical studies indicate that, when problem data is randomly generated to satisfy Assumption \ref{asmp1}, (\ref{lineqsys}) is usually a well-conditioned linear system.

\subsection{Step 3: Solving Lyapunov equations}
Finally, $\check{X}_i$ for $i\in\{2,3,\cdots,N\}$ and $\check{Y}_i$ for $i\in\{1,2,\cdots,N-1\}$ are computed by looking at $(\cdot)_{\uparrow i-1}^{\uparrow i-1}$ sub-block of (\ref{rcicr1}) and $(\cdot)_{\downarrow i+1}^{\downarrow i+1}$ sub-block of (\ref{rcifr1}).
\begin{subequations}
\label{lyapsys}
\begin{align}
&{\hat{A}_{i-1}^L\!}^T\check{X}_i+\check{X}_i\hat{A}_{i-1}^L+
{\bar{A}_{i-1}^L\!}^T\bar{X}_i^T+\bar{X}_i\bar{A}_{i-1}^L-
\bar{K}_i^T\Psi_{\downarrow i}^{\downarrow i}\bar{K}_i \nonumber  \\
&+\text{row}\{\bar{K}_{i-1}, {\hat{K}_{i-1}^a} \}^T\Psi_{\downarrow i-1}^{\downarrow i-1}\text{row}\{\bar{K}_{i-1}, {\hat{K}_{i-1}^a} \}=0 \nonumber \\
& i\in\{2,3,\cdots,N\}\\
&\hat{A}_{i+1}^K\check{Y}_i+\check{Y}_i{\hat{A}_{i+1}^K\!}^T+
\bar{A}_{i+1}^K\bar{Y}_i+\bar{Y}_i^T{\bar{A}_{i+1}^K\!}^T-\bar{L}_i\Phi_{\uparrow i}^{\uparrow i}\bar{L}_i^T \nonumber \\
&+\text{col}\{ \hat{L}_{i+1}^a, \bar{L}_{i+1} \}
\Phi_{\uparrow i+1}^{\uparrow i+1}\text{col}\{ \hat{L}_{i+1}^a, \bar{L}_{i+1} \}^T=0 \nonumber \\ & i\in\{1,2,\cdots,N-1\}
\end{align}
\end{subequations}
Since all other quantities are known by the previous step, these are Lyapunov equations with respect to $\check{X}_i$ and $\check{Y}_i$, which can be easily solved.
Since $\hat{A}_{i}^K$ and  $\hat{A}_{i}^L$ are Hurwitz stable (guaranteed by Proposition \ref{propstabsol}), they admit a unique solution.

\begin{proposition}
\label{proprcsol}
Under Assumption \ref{asmp1} and \ref{asmp2}, the set of algebraic Riccati equations (\ref{coupledrc}) admit a unique tuple of positive semidefinite solutions $X_i, Y_i$, $i\in\{1,2,\cdots,N\}$. Moreover, they are stabilizing solutions.
\end{proposition}
\begin{proof}
It is clear from Theorem \ref{theorceq} that (\ref{rc1c}) and (\ref{rcNf}) admit unique positive semidefinite solutions which are stabilizing. To see why solutions constructed in Step 1, 2 and 3 above are positive semidefinite, notice that under Assumption \ref{asmp1} and \ref{asmp2}, the algorithm produces a unique set of variables satisfying (\ref{rcicr1}) and (\ref{rcifr1}). 
Furthermore, notice that
$$
A_{i,i-1}^{KL}=\left[ \begin{array}{cc} 
\hat{A}_{i-1}^L & 0  \\
A_{\downarrow  i}^{\uparrow i-1}+B_{\downarrow i}^{\downarrow i}\bar{K}_i+\bar{L}_{i-1}C_{\uparrow i-1}^{\uparrow i-1} &
\hat{A}_{i}^K
\end{array}\right]
$$
is a stable matrix since its diagonal blocks are stable. 
Hence, due to the inertia property of a Lyapunov equation, $X_i$ and $Y_i$ must be positive semidefinite. They are indeed stabilizing solutions since $A_{i,i-1}^{KL}$ is a stable matrix.
\end{proof}

\section{Derivation of Main Result}
\label{secopt}
\subsection{Stability}
Notice that the closed-loop transfer function is given by
\begin{align}
\mathbf{G}^{cl}&=
\left[ \begin{array}{cc}
\mathbf{G}^{cl}_{11} & \mathbf{G}^{cl}_{12} \\
\mathbf{G}^{cl}_{21} & \mathbf{G}^{cl}_{22} \end{array}\right]
=\left[ \begin{array}{c|cc}
\mathcal{A} & \mathcal{W} & \mathcal{B} \\ \hline  
\mathcal{F} & 0 & H \\
\mathcal{C} & V & 0
\end{array}\right] \nonumber \\
&\triangleq
\left[ \begin{array}{cc|cc} 
A_K & B_KC & B_KV & 0 \\
BC_K & A & W& B\\ \hline 
HC_K & F & 0 & H \\
0 & C & V & 0
\end{array}\right]. \label{cltf}
\end{align}
To see that $\mathbf{K}_{opt}$ is a stabilizing controller, we need to verify that $\mathcal{A}$ is a stable matrix.
Let $\bar{\zeta}$ be an $n(N+1)$-dimensional square matrix whose $(i,j)$-th sub-block is $I_n$ if $i\geq j$ and is zero otherwise. Also define $\bar{\mu}=\bar{\zeta}^{-1}$.
One can easily check that a similarity transformation gives
% \begin{equation}
% \label{muAzeta}
% \mu\mathcal{A}\zeta=
% \left[ \begin{array}{ccccc}
% A\!+\!B^{\downarrow 1}K_1 \!&\! -L_1C_{\uparrow 1} \!&\! \cdots \!&\!-L_1C_{\uparrow 1} \!&\! -L_1C_{\uparrow 1} \\
% \!&\! A\!+\!B^{\downarrow 2}K_2\!+\!L_1C_{\uparrow 1} \!&\! \cdots \!&\!L_1C_{\uparrow 1}\!-\!L_2C_{\uparrow 2} \!&\! L_1C_{\uparrow 1}\!-\!L_2C_{\uparrow 2}  \\ [-1ex] 
% \!&\!\!&\! \ddots \!&\! \vdots \!&\! \vdots  \\
% \!&\!\!&\!\!&\! A\!+\!B^{\downarrow N}K_N\!+\!L_{N-1}C_{\uparrow N-1} \!&\!L_{N-1}C_{\uparrow N-1}\!-\!L_NC_{\uparrow N}  \\
% 0\!&\!\!&\!\!&\!\!&\! A\!+\!L_NC_{\uparrow N} \end{array}\right].
% \end{equation}
\begin{equation}
\label{muAzeta}
\bar{\mu}\mathcal{A}
\bar{\zeta}
=
\left[ \begin{array}{ccccc}
A^{KL}_{1,0} \!&\! * \!&\! * \!&\!* \!&\! * \\
\!&\! A^{KL}_{2,1} \!&\! * \!&\!* \!&\! *  \\ [-1ex] 
\!&\!\!&\! \ddots \!&\! * \!&\! *  \\
\!&\!\!&\!\!&\! A^{KL}_{N,N-1} \!&\!* \\
0\!&\!\!&\!\!&\!\!&\! A^{KL}_{N+1,N} \end{array}\right].
\end{equation}
This is a block upper-triangular matrix. Moreover, all diagonal blocks are stable matrices as we saw in the proof of Proposition \ref{proprcsol}. This shows the stability of $\mathcal{A}$.

\subsection{Optimal Controller Characterization}
\label{secchar}

\begin{figure}[t]
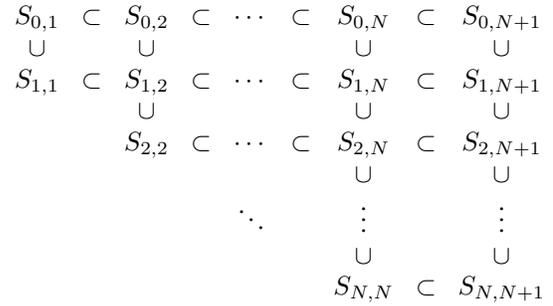

\centering
\begin{equation*}
\begin{array}{ccccccccc}
S_{0,1} & \!\subset\! & S_{0,2} & \!\subset\! & \cdots  & \!\subset\! & S_{0,N} & \!\subset\! & S_{0,N+1} \\
\cup & & \cup & & &  & \cup & & \cup  \\
S_{1,1} & \!\subset\! & S_{1,2} & \!\subset\! & \cdots & \!\subset\! & S_{1,N} & \!\subset\! & S_{1,N+1} \\
 & & \cup & & &  & \cup & & \cup  \\
 & & S_{2,2} & \!\subset\! & \cdots &  \!\subset\! & S_{2,N} & \!\subset\! & S_{2,N+1} \\
 & & & & &  & \cup & & \cup \\
 & & & & \ddots &   & \vdots &  & \vdots \\
 & &  & & &  & \cup & & \cup \\
&&&&&&   S_{N,N} & \!\subset\! & S_{N,N+1}
 \end{array}
\end{equation*}
\caption{Inclusion relation diagram among subspaces $S_{i,j}$ defined by (\ref{defsij}). For convenience, we also define $S_{0,j}\triangleq \{\mathbf{Q}E_{\uparrow j}\mathbf{G}_{21}^{cl}:\mathbf{Q}\in H_2\}$, $S_{i,N+1}\triangleq \{\mathbf{G}_{12}^{cl}E^{\downarrow i}\mathbf{Q}:\mathbf{Q}\in H_2\}$, $S_{0,N+1}\triangleq H_2$.}
\label{figdiagram}
\vspace{.1in}
\end{figure}

\begin{figure*}[t]
\centering
\begin{align}
&\mathcal{K}_i\triangleq E_{\downarrow i}\mbox{row}\{0,\cdots,0,\overset{i\text{-th block}}{E^{\downarrow i+1}K_{i+1}-E^{\downarrow i}K_i},\cdots,\overset{(N-1)\text{-th block}}{E^{\downarrow N}K_N-E^{\downarrow N-1}K_{N-1}}, \overset{N\text{-th block}}{-E^{\downarrow N}K_N},\overset{(N+1)\text{-th block}}{E^{\downarrow i}K_i}\} \label{calK}\\
&\mathcal{L}_j\triangleq \mbox{col}\{\overset{1\text{st block}}{L_1E_{\uparrow  1}},\cdots,\overset{(j-1)\text{-th block}}{L_{j-1}E_{\uparrow j-1}}, \overset{j\text{-th block}}{L_{j}E_{\uparrow j}},\cdots,\overset{N\text{-th block}}{L_{j}E_{\uparrow j}},\overset{(N+1)\text{-th block}}{L_{j}E_{\uparrow j}}\}E^{\uparrow j} \label{calL} \\
&\tilde{J}_i\triangleq \mbox{row}\{0,\cdots,0,\overset{(i-1)\text{-th block}}{I}, 0,\cdots,0,\overset{(N+1)\text{-th block}}{-I}\}, \hat{J}_j\triangleq \mbox{col}\{0,\cdots,0,\overset{(j+1)\text{-th block}}{I},\cdots,\overset{(N+1)\text{-th block}}{I}\} \label{calJ}
\end{align}
\vspace{-.07in}
\hrule
\begin{align}
\label{defui}
\mathbf{U}_i&\!\triangleq \!\left[\! \begin{array}{c|c}
A^{KL}_{i,i-1} \!&\!  B_{U_i} \\ \hline
C_{U_i} \!&\! D_{U_i}\end{array}\!\right], \;
\mathbf{M}_i^{-1}\!\triangleq\!\left[\! \begin{array}{c|c} 
\mathcal{A} \!&\! \mathcal{B}^{\downarrow i} \\ \hline
-{\Psi_{\downarrow i}^{\downarrow i}}^\frac{1}{2}\mathcal{K}_i \!&\! {\Psi_{\downarrow i}^{\downarrow i}}^\frac{1}{2}\end{array}\!\right] \forall i\in\{1,2,\cdots,N\} 
 \text{ where } B_{U_i}=B^{\downarrow i}{\Psi_{\downarrow i}^{\downarrow i}}^{-\frac{1}{2}}, \\
& C_{U_i}\!\!=\!\! \begin{cases} 
F+HK_1 &\mbox{if } i=1 \\
{\Psi_{\downarrow i-1}^{\downarrow i-1}}^\frac{1}{2}E_{\downarrow i-1}(E^{\downarrow i}K_i\!\!-\!\!E^{\downarrow i-1}K_{i-1}) &\mbox{if } 2\leq i \leq N  \end{cases}, \;
D_{U_i} \!\!=\!\! \begin{cases} 
H\Psi^{-\frac{1}{2}} &\mbox{if } i=1 \\
{\Psi_{\downarrow i-1}^{\downarrow i-1}}^\frac{1}{2}E_{\downarrow i-1}E^{\downarrow i}{\Psi_{\downarrow i}^{\downarrow i}}^{-\frac{1}{2}} &\mbox{if } 2\leq i \leq N.  \end{cases} \nonumber \\
\label{defvj}
\mathbf{V}_j&\!\triangleq\!\left[ \begin{array}{c|c}
A^{KL}_{j+1,j} \!&\!  B_{V_j} \\ \hline
C_{V_j} \!&\! D_{V_j}\end{array}\right], \;
\mathbf{N}_i^{-1}\!\triangleq\!\left[ \begin{array}{c|c} 
\mathcal{A} \!&\! -\mathcal{L}_j{\Phi_{\uparrow j}^{\uparrow j}}^{\frac{1}{2}} \\ \hline
\mathcal{C}_{\uparrow j} \!&\! {\Phi_{\uparrow j}^{\uparrow j}}^{\frac{1}{2}}\end{array}\right]
\forall j\in\{1,2,\cdots,N\} \text{ where } C_{V_j}={\Phi_{\uparrow  j}^{\uparrow  j}}^{-\frac{1}{2}}C_{\uparrow j}, \\
&B_{V_j} \!\!=\!\! \begin{cases} 
W+L_NV &\mbox{if } j=N \\
(L_jE_{\uparrow j}\!\!-\!\!L_{j+1}E_{\uparrow j+1})E^{\uparrow j+1}{\Phi_{\uparrow j+1}^{\uparrow j+1}}^{\frac{1}{2}} & \!\mbox{if } 1\! \leq \! j \!\leq \! N\!-\!1 \end{cases}, 
D_{V_j} \!\!=\!\! \begin{cases} 
\Phi^{-\frac{1}{2}}V & \!\mbox{if } i=N \\
{\Phi_{\uparrow j}^{\uparrow j}}^{-\frac{1}{2}}E_{\uparrow j}E^{\uparrow j+1}{\Phi_{\uparrow j+1}^{\uparrow j+1}}^{\frac{1}{2}} &\mbox{if } 1\leq i \leq N-1.  \end{cases} \nonumber
\end{align}
\vspace{-.07in}
\hrule
\end{figure*}

If $\mathbf{K}_{opt}$ is the optimal solution to Problem \ref{h2controlproblem}, then any perturbation $\mathbf{K}=\mathbf{K}_{opt}+\mathbf{K}'$ such that $\mathbf{K}'\in\mathcal{I}_{LBT}$ only degrades control performance. 
Due to the uniqueness of the rational solution to Problem \ref{h2controlproblem}, it is sufficient for us to show that $\mathbf{K}'=0$ is the optimal solution to the modified $H_2$ optimal control problem
\begin{align*}
\min & \;\; \|\mathbf{G}^{cl}_{11}+\mathbf{G}^{cl}_{12}\mathbf{K}'(I-\mathbf{G}^{cl}_{22}\mathbf{K}')^{-1}\mathbf{G}^{cl}_{21}\| \\
\text{s.t. } & \;\; \mathbf{K}' \text{ is stabilizing and } \mathbf{K}'\in \mathcal{I}_{LBT}.
\end{align*}
Since $\mathbf{G}^{cl}_{22}\in H_2 \cap \mathcal{I}_{LBT}$, the subspace $\mathcal{I}_{LBT}$ is quadratically invariant under $\mathbf{G}^{cl}_{22}$. 
This means that all stabilizing controllers are parametrized by the structured Youla parameter $\mathbf{Q}\triangleq -\mathbf{K}'(I-\mathbf{G}^{cl}_{22}\mathbf{K}')^{-1} \in H_\infty\cap \mathcal{I}_{LBT}$.
Hence, the above statement is equivalent to that $\mathbf{Q}=0$ is the optimal solution to the model matching problem\footnote{The condition $\mathbf{Q} \in H_\infty \cap \mathcal{I}_{LBT}$ can be replaced by $\mathbf{Q} \in H_2 \cap \mathcal{I}_{LBT}$ without loss of generality. Recall that under Assumption \ref{asmp1}, $H$ has full column rank, $V$ has full row rank, and $\mathbf{G}_{11}^{cl}\in H_2$. This means that $\mathbf{Q}$ must be in $H_2$ so that the value of (\ref{qlbtptb}) is bounded.}
\begin{align}
\min & \;\; \|\mathbf{G}^{cl}_{11}-\mathbf{G}^{cl}_{12}\mathbf{Q}\mathbf{G}^{cl}_{21}\| \label{qlbtptb} \\
\text{s.t.} & \;\; \mathbf{Q} \in H_2 \cap \mathcal{I}_{LBT}. \nonumber
\end{align}
Since the non-rectangular constraint $\mathcal{I}_{LBT}$ is inconvenient to work with, we use an alternative characterization of the same statement using rectangular blocks $\mathcal{I}^{\uparrow i}_{\downarrow i}$, $i\in\{1,2,\cdots,N\}$.
\begin{proposition}
\label{prop1}
$\mathbf{Q}=0$ is the optimal solution to (\ref{qlbtptb}) if and only if for every $i\in\{1,2,\cdots,N\}$, $\mathbf{Q}_i=0$ is the optimal solution to the model matching problem
\begin{align}
\min & \;\; \|\mathbf{G}^{cl}_{11}-\mathbf{G}^{cl}_{12}\mathbf{Q}_i\mathbf{G}^{cl}_{21}\| \label{modelmatchi}  \\
\text{s.t.} & \;\; \mathbf{Q}_i \in H_2 \cap \mathcal{I}^{\uparrow i}_{\downarrow i}. \nonumber 
\end{align}
\end{proposition}
% \begin{proof}
% Necessity follows since $\mathcal{I}^{\uparrow i}_{\downarrow i}\subset \mathcal{I}_{LBT}$ for every $i\in\{1,2,\cdots,N\}$.
% To see sufficiency, notice that the latter condition of the statement implies that
% $$
% \left<\mathbf{G}_{11}^{cl},\mathbf{G}_{12}^{cl}\mathbf{Q}'_{ij}\mathbf{G}_{21}^{cl}\right>=0 \;\; \forall \mathbf{Q}'_{ij}\in H_2\cap \mathcal{I}_i^j
% $$
% for every $i\geq j$ due to the optimality condition. Hence for every $\mathbf{Q}' \in H_2\cap \mathcal{I}_{LBT}$,
% $$
% \left<\mathbf{G}_{11}^{cl},\mathbf{G}_{12}^{cl}\mathbf{Q}'\mathbf{G}_{21}^{cl}\right>=\sum_{i\geq j} 
% \left<\mathbf{G}_{11}^{cl},\mathbf{G}_{12}^{cl}\mathbf{Q}'_{ij}\mathbf{G}_{21}^{cl}\right>=0.
% $$
% This is the optimality condition for $\mathbf{Q}=0$ to be the solution to (\ref{qlbtptb}).
% \end{proof}

One approach to find a solution to the model matching problem (\ref{modelmatchi}) is to apply the projection theorem \cite{RefWorks:295}. 
% Let $H$ be a Hilbert space and $S$ be a closed nonempty subspace of $H$. If $a\in H$, then there exists an $x_0\in S$ such that $\|a-x_0\|\leq \|a-x\|$ for all $x\in S$. Such $x_0$ satisfies $\left<a-x_0,x\right>=0$ for all $x\in S$. Moreover, if an element $x_0\in S$ satisfies $\left<a-x_0,x\right>=0$ for all $x\in S$, then $\|a-x_0\|\leq \|a-x\|$ for all $x\in S$. Such an element $x_0$ is called a projection of $a$ onto $S$, and is denoted by $x_0=\pi_S(a)$.
Define subspaces $S_{i,j}$ of $H_2$ for $1\leq i \leq j \leq N$ by
\begin{equation}
\label{defsij}
S_{i,j}\triangleq \{\mathbf{G}_{12}^{cl}E^{\downarrow i}\mathbf{Q}E_{\uparrow i}\mathbf{G}_{21}^{cl}: \mathbf{Q}\in H_2 \}.
\end{equation}
Proposition \ref{prop1} implies that, in order to infer that $\mathbf{K}_{opt}$ is the optimal controller, it suffices to prove that $\mathbf{Q}=0$ is the minimizer of $\|\mathbf{G}_{11}^{cl}-\mathbf{G}_{12}^{cl}E^{\downarrow i}\mathbf{Q}E_{\uparrow i}\mathbf{G}_{21}^{cl}\|$ over $\mathbf{Q}\in H_2$ for every $i\in\{1,2,\cdots,N\}$. Equivalently, it needs to be shown that $\pi_{S_{i,i}}(\mathbf{G}_{11}^{cl})=0$ for every $i\in\{1,2,\cdots,N\}$, where $\pi_{S_{i,j}}:H_2\rightarrow S_{i,j}$ is the projection operator.

\subsection{Nested Projections}
\label{secproj}
It is clear that the inclusion relations in Fig. \ref{figdiagram} hold among subspaces defined by (\ref{defsij}). We are going to exploit this diagram to find an explicit representation of $\pi_{S_{i,j}}(\mathbf{G}_{11}^{cl})$. Recall the following fact:
% In principle, it is possible to compute a projection $\pi_{S_{i,i}}(\mathbf{G}_{11})$ by applying a similar technique to the inner-outer factorizations to $\mathbf{G}_{12}^{cl}E^{\downarrow i}$ and $E_{\uparrow i}\mathbf{G}_{21}^{cl}$ as in Section \ref{seccent2}. 
% However, we are not aware of a simple formula to perform this projection in a single shot. 
% Alternatively, we consider doing it in multiple stages. 
\begin{theorem}(Nested Projections, see e.g., \cite{RefWorks:307})
\label{theonproj}
Let $S_1, S_2,\cdots, S_N$ be subspaces of a Hilbert space such that $S_N\subset S_{N-1}\subset \cdots S_2 \subset S_1$. Then $\pi_{S_N}=\pi_{S_N}\circ \pi_{S_{N-1}}\circ \cdots \circ \pi_{S_2} \circ \pi_{S_1}$.
\end{theorem}
According to Fig. \ref{figdiagram}, Theorem \ref{theonproj} suggests that $\pi_{S_{i,i}}$ can be computed as, for instance,
\begin{equation}
\pi_{S_{i,i}}=\pi_{S_{i,i}}\circ \pi_{S_{i,i+1}} \circ \cdots \circ \pi_{S_{i,N}} \circ \pi_{S_{i-1,N}} \circ \cdots \circ \pi_{S_{1,N}}. \label{sijsuccessive}
\end{equation}
Understanding $\pi_{S_{i,i}}$ as a composition of stepwise projections is convenient in the following presentation, since each projection step can be associated with one of $2N$ Riccati equations in (\ref{coupledrc}). To be precise, we consider writing $S_{i,j}$ using an ``orthonormal" basis. 
Recall that a rational function $\mathbf{U}\in H_\infty$ is said to be \emph{inner} if $\mathbf{U}^*\mathbf{U}=I$ and \emph{co-inner} if $\mathbf{U}\mathbf{U}^*=I$.
It turns out that each subspace can be written as 
\begin{equation}
\label{sijuv}
S_{i,j}=\{\mathbf{U}_1\cdots \mathbf{U}_i\mathbf{M}_i^{-1}\mathbf{Q}\mathbf{N}_j^{-1}\mathbf{V}_j\cdots \mathbf{V}_N: \mathbf{Q}\in H_2\}
\end{equation}
where the explicit form of inner functions $\mathbf{U}_1,\cdots,\mathbf{U}_N$, co-inner functions $\mathbf{V}_1,\cdots, \mathbf{V}_N$ and other necessary quantities are given in (\ref{calK})-(\ref{defvj}). The above expression is obtained by repeated applications of a particular type of spectral factorizations (Lemma \ref{lemUM} in Appendix \ref{proof:lempij}). Each application of the factorization requires a solution to one of the Riccati equations in (\ref{coupledrc}).

Writing $S_{i,j}$ in the form of (\ref{sijuv}) makes nested projections easier.
Suppose that the projection of $\mathbf{G}_{11}^{cl}$ onto $S_{i,j}$ can be written in the form of
$$
\pi_{S_{i,j}}(\mathbf{G}_{11}^{cl})=\mathbf{U}_1\cdots \mathbf{U}_i\tilde{\mathbf{P}}_i \mathbf{V}_j\cdots \mathbf{V}_N \in S_{i,j}
$$
for some $\tilde{\mathbf{P}}_i$. Then it is easy to check that the subsequent projection is given by 
$$
\pi_{S_{i+1,j}}(\mathbf{G}_{11}^{cl})=\mathbf{U}_1\cdots \mathbf{U}_{i+1}\tilde{\mathbf{P}}_{i+1} \mathbf{V}_j\cdots \mathbf{V}_N \in S_{i+1,j}
$$
where $\tilde{\mathbf{P}}_{i+1}$ is chosen to satisfy the optimality condition 
$$
\left<\mathbf{U}_{i+1}^*\tilde{\mathbf{P}}_i-\tilde{\mathbf{P}}_{i+1}, \mathbf{M}_{i+1}^{-1}\mathbf{Q}\mathbf{N}_{j}^{-1}\right>=0 \; \forall \mathbf{Q}\in H_2.
$$
Details can be found in Lemma \ref{lemproj} in Appendix \ref{proof:lempij}. Also, notice that every projection generates a ``residual term" as
\begin{align*}
&\pi_{S_{i,j}}\!(\mathbf{G}_{11}^{cl})\!\!=\!\!\pi_{S_{i+1,j}}\!(\mathbf{G}_{11}^{cl})+\mathbf{U}_1\!\cdots\! \mathbf{U}_i\mathbf{R}_{(i,j)\rightarrow (i+1,j)}\mathbf{V}_j\!\cdots\! \mathbf{V}_N \\
&\pi_{S_{i,j}}\!(\mathbf{G}_{11}^{cl})\!\!=\!\!\pi_{S_{i,j-1}}\!(\mathbf{G}_{11}^{cl})+\mathbf{U}_1\!\cdots\! \mathbf{U}_i\mathbf{R}_{(i,j)\rightarrow (i,j-1)}\mathbf{V}_j\!\cdots\! \mathbf{V}_N.
\end{align*}
The $H_2$ norm of residual terms will be used later to compute the optimal value of Problem \ref{h2controlproblem}. Finally, all the above operations can be performed at the state space level, as summarized in Lemma \ref{lempij}.
\begin{figure*}[t]
\centering
\vspace{-.10in}
\begin{align}
&\|\mathbf{R}_{(i-1,j)\rightarrow (i,j)}\|^2=
\begin{cases}
tr W^TX_1 W &\mbox{if } i=1, j=1 \\
tr \Phi_{\uparrow j}^{\uparrow j} L_j^TX_1L_j &\mbox{if } i=1, j\geq 2 \\
tr (L_{i-1}V_{\uparrow i-1}-L_jV_{\uparrow j})^TX_1(L_{i-1}V_{\uparrow i-1}-L_jV_{\uparrow j}) &\mbox{if } i\geq 2 
\end{cases} \label{resnorm1} \\
&\|\mathbf{R}_{(i,j+1)\rightarrow (i,j)}\|^2=
\begin{cases}
tr F Y_N F^T &\mbox{if } i=N, j=N \\
tr \Psi_{\downarrow i}^{\downarrow i} K_iY_NK_i^T &\mbox{if } i\leq N-1, j=N \\
tr (H^{\downarrow i}K_i-H^{\downarrow j+1}K_{j+1})Y_j(H^{\downarrow i}K_i-H^{\downarrow j+1}K_{j+1})^T &\mbox{if } j\leq N-1.
\end{cases} \label{resnorm2}
\end{align}
\vspace{-.07in}
\hrule
\end{figure*}

\begin{lemma}
\label{lempij}
The projection of $\mathbf{G}_{11}^{cl}$ onto any subspace $S_{i,j}$ in Fig. \ref{figdiagram} is given by 
$
\pi_{S_{i,j}}(\mathbf{G}_{11}^{cl})=\mathbf{U}_1\mathbf{U}_2\cdots \mathbf{U}_i\mathbf{P}'_{i,j}\mathbf{V}_j\cdots \mathbf{V}_{N-1}\mathbf{V}_N
$
where
\begin{equation}
\label{p_ij}
\mathbf{P}'_{i,j}=\left[ \begin{array}{c|c}
\mathcal{A} & \Lambda_j \\ \hline 
 \Gamma_i & 0
\end{array}\right]
\end{equation}
$$
\Gamma_i \!\!=\!\!
\begin{cases} 
\mathcal{F} &\!\!\mbox{if } i=0 \\
\!-{\Psi_{\downarrow i}^{\downarrow i}}^{\frac{1}{2}}\mathcal{K}_i &\!\!\mbox{if } 1\!\leq\! i \!\leq\! N 
\end{cases}\!\!, 
\Lambda_j\!\!=\!\!
\begin{cases} 
\mathcal{W} &\!\mbox{if } j\!=\!N\!+\!1 \\
\!-\mathcal{L}_j {\Phi_{\uparrow j}^{\uparrow j}}^{\frac{1}{2}} &\!\mbox{if } 1 \!\leq\! j \!\leq\! N 
\end{cases}.
$$
Moreover,
\begin{align*}
\mathbf{R}_{(i-1,j)\rightarrow (i,j)}&=
\left[ \begin{array}{c|c}
A_{i,i-1}^{KL} & -\tilde{J}_i \Lambda_j \\
 \hline 
C_{U_i} & 0
\end{array}\right] \\
\mathbf{R}_{(i,j+1)\rightarrow (i,j)}&=
\left[ \begin{array}{c|c}
A_{j+1,j}^{KL} & B_{V_j} \\ \hline
-\Gamma_i\hat{J}_j & 0
\end{array}\right]. 
\end{align*}
\end{lemma}
\begin{proof} See Appendix \ref{proof:lempij}. \end{proof}

\subsection{Proof of Optimality}
We are now ready to prove that  $\pi_{S_{i,i}}(\mathbf{G}_{11}^{cl})=0$ for every $i\in\{1,2,\cdots,N\}$. 
Combined with Proposition \ref{prop1}, this completes the proof of optimality of the proposed controller.

\begin{proof} (of Theorem \ref{maintheo})
We have verified the existence and uniqueness of the stabilizing solution to (\ref{coupledrc}) in Section \ref{seccoupledrc}. By Lemma \ref{lempij}, for every $i\in\{1,2,\cdots,N\}$, we have
$$
\mathbf{P}'_{i,i}
=\left[ \begin{array}{c|c}
\mathcal{A} & -\mathcal{L}_i {\Phi_{\uparrow i}^{\uparrow i}}^{\frac{1}{2}} \\ \hline
-{\Psi_{\downarrow i}^{\downarrow i}}^{\frac{1}{2}}\mathcal{K}_{i} & 0
\end{array}\right].
$$
Apply a state space transformation defined by $\bar{\mu}$ and $\bar{\zeta}$. As we have observed in (\ref{muAzeta}), $\bar{\mu}\mathcal{A}\bar{\zeta}$ is an upper block triangular matrix. 
Also, it is straightforward to check that all $(\cdot)_{\downarrow i+1}$ sub-blocks of $\bar{\mu}\mathcal{L}_i{\Phi_{\uparrow i}^{\uparrow i}}^{\frac{1}{2}}$ are zero. Furthermore, it is possible to show that all $(\cdot)^{\uparrow i}$ sub-blocks of ${\Psi_{\downarrow i}^{\downarrow i}}^{\frac{1}{2}}\mathcal{K}_{i}\bar{\zeta}$ are zero. Hence $\mathbf{P}'_{i,i}=0$.
Therefore, for every $i\in\{1,2,\cdots,N\}$, $\pi_{S_{i,i}}(\mathbf{G}_{11}^{cl})=\mathbf{U}_1\mathbf{U}_2\cdots\mathbf{U}_i\mathbf{P}'_{i,i}\mathbf{V}_1\cdots\mathbf{V}_{N-1}\mathbf{V}_N=0$.
By Proposition \ref{prop1}, this implies that the proposed controller is the optimal solution to Problem \ref{h2controlproblem}.

Since we have shown that $\mathbf{G}_{11}^{cl}$ is the optimal closed loop transfer function, the optimal cost is given by computing its $H_2$ norm. 
To obtain more explicit expression, consider a nested projection $\pi_{S_{1,1}}=\pi_{S_{1,1}}\circ \cdots \circ \pi_{S_{1,N}}\circ \pi_{S_{1,N+1}}$. It is possible to write
\begin{align*}
\mathbf{G}_{11}^{cl}=&\pi_{S_{1,1}}\!(\mathbf{G}_{11}^{cl})\!+\!\mathbf{R}_{(0,N+1)\rightarrow (1,N+1)}\!+\!\mathbf{U}_1\mathbf{R}_{(1,N+1)\rightarrow (1,N)} \\
&+\sum_{j=1}^{N-1}\mathbf{U}_1\mathbf{R}_{(1,j+1)\rightarrow (1,j)}\mathbf{V}_{j+1}\cdots\mathbf{V}_N.
\end{align*}
Since $\pi_{S_{1,1}}(\mathbf{G}_{11}^{cl})=0$ and all residual terms are orthogonal, the optimal cost $J_{opt}=\|\mathbf{G}_{11}^{cl}\|$ can be decomposed as
\begin{align*}
\|\mathbf{G}_{11}^{cl}\|^2=&\|\mathbf{R}_{(0,N+1)\rightarrow (1,N+1)}\|^2+\|\mathbf{R}_{(1,N+1)\rightarrow (1,N)}\|^2 \\
&+\sum_{j=1}^{N-1}\|\mathbf{R}_{(1,j+1)\rightarrow (1,j)}\|^2.
\end{align*}
Each term can be written more explicitly using the fact (\ref{resnorm1}) (\ref{resnorm2}). 
This proves $J_{opt}^2=J_{cnt}^2+J_{dcnt}^2$.
\end{proof}

% \section{Numerical Example}
% \label{secnum}
% \begin{tabular}{cccc}
% \hline
% Controller structure & State space dimension of & Closed loop & Condition number $\sigma_{max}/\sigma_{min}$ \\
% (Size of diagonal blocks) & the optimal controller & $H_2$ performance & of the linear system (\ref{lineqsys}) \\
% \hline
% 1,1,1,1              & 16 & 21.82 & 148.7 \\
% 2,1,1            & 12 & 21.74 & 146.6 \\
% 2,2          & 8 & 20.55 & 54.0 \\
% 4 (centralized)     & 4 & 13.83 & -  \\
% \hline
% \end{tabular}

\section{Conclusion and Future Work}
In this paper, we have presented a state-space realization of the optimal output feedback controller for the $N$-player triangular LQG problem. We have derived a set of algebraic Riccati equations to be solved to construct the optimal controller.
Solvability of Riccati equations, namely non-singularity of the linear system (\ref{lineqsys}), must be verified in the future work.

% Although the proposed controller (\ref{akbkck}) has an attractive structure, we are aware that the optimal LQG controllers for general poset causal systems are not in this form, even if $\zeta$ and $\mu$ are chosen accordingly (a counterexample is \cite{RefWorks:290}). Overall, the state-of-art results on the explicit solutions to the optimal distributed control problems, even if the information constraints are all quadratically invariant, are rather sporadic, and appear different depending on the problem type (e.g., information structure, control objective). 
% Hence in the future, an effort has to be taken towards the unified understanding of  the optimal distributed controller architectures.

\bibliographystyle{plain}
\bibliography{Poset}

\begin{thebibliography}{10}

\bibitem{RefWorks:271}
Y.-C. Ho and K.-C. Chu.
\newblock Team decision theory and information structures in optimal control
  problems--part i.
\newblock {\em IEEE Transactions on Automatic Control}, 17(1):15--22, 1972.

\bibitem{RefWorks:273}
A.~Lamperski and J.~C. Doyle.
\newblock Dynamic programming solutions for decentralized state-feedback {LQG}
  problems with communication delays.
\newblock In {\em American Control Conference (ACC), 2012}, pages 6322--6327,
  2012.

\bibitem{RefWorks:267}
L.~Lessard.
\newblock Decentralized {LQG} control of systems with a broadcast architecture.
\newblock In {\em Decision and Control (CDC), 2012 IEEE 51st Annual Conference
  on}, pages 6241--6246, 2012.

\bibitem{RefWorks:294}
L.~Lessard and S.~Lall.
\newblock Optimal control of two-player systems with output feedback, 2013.
\newblock arXiv:1303.3644.

\bibitem{RefWorks:295}
D.~G. Luenberger.
\newblock {\em Optimization by Vector Space Methods}.
\newblock Wiley, 1968.
\newblock 68008716.

\bibitem{RefWorks:307}
D.~G. Luenberger.
\newblock Projection pricing.
\newblock {\em Journal of Optimization Theory and Applications}, 109(1):1--25,
  2001.

\bibitem{RefWorks:17}
M.~Rotkowitz and S.~Lall.
\newblock A characterization of convex problems in decentralized control.
\newblock {\em IEEE Transactions on Automatic Control}, 50(12):1984--1996,
  2005.

\bibitem{RefWorks:293}
M.~Rotkowitz and S.~Lall.
\newblock Convexification of optimal decentralized control without a
  stabilizing controller.
\newblock In {\em Proceedings of the International Symposium on Mathematical
  Theory of Networks and Systems (MTNS)}, pages 1496--1499, 2006.

\bibitem{RefWorks:277}
S.~Sab\u{a}u and N.~C. Martins.
\newblock On the stabilization of {LTI} decentralized configurations under
  quadratically invariant sparsity constraints.
\newblock In {\em 48th Annual Allerton Conference on Communication, Control,
  and Computing}, pages 1004--1010, 2010.

\bibitem{RefWorks:265}
P.~Shah and P.A. Parrilo.
\newblock An optimal controller architecture for poset-causal systems, 2011.
\newblock arXiv:1111.7221.

\bibitem{RefWorks:176}
S.~Skogestad and I.~Postlethwaite.
\newblock {\em Multivariable feedback control: analysis and design}.
\newblock John Wiley, 2005.

\bibitem{RefWorks:269}
J.~Swigart and S.~Lall.
\newblock Optimal controller synthesis for a decentralized two-player system
  with partial output feedback.
\newblock In {\em American Control Conference (ACC), 2011}, pages 317--323,
  2011.

\bibitem{RefWorks:292}
M.~Vidyasagar.
\newblock {\em Control system synthesis: a factorization approach}.
\newblock MIT Press, 1985.
\newblock 84014411.

\bibitem{RefWorks:21}
K.~Zhou, J.~C. Doyle, and K.~Glover.
\newblock {\em Robust and optimal control}.
\newblock Prentice-Hall, Inc, Upper Saddle River, NJ, USA, 1996.

\end{thebibliography}

\appendix

\ifdefined\LONGVERSION
\subsection{Quadratic Invariance and Convexity}
\label{secqi}
This section gives a brief review of the notion of \emph{Quadratic Invariance} and how an optimal control problem can be formulated as an infinite dimensional convex optimization problem. For brevity, our discussion here is rather informal; for a thorough introduction, readers are referred to \cite{RefWorks:17}.
The $H_2$ optimal control formulated in Problem \ref{h2controlproblem} is a nonconvex optimization problem respect to $\mathbf{K}$.
A natural approach is to introduce the \emph{Youla parametrization}, which has been historically used to convexify centralized control problems.
Youla parameterization is particularly simple if $\mathbf{G}$ is stable, and is given by
$\mathbf{Q}=-\mathbf{K}(I-\mathbf{G}_{22}\mathbf{K})^{-1}$.
The inverse in this expression is guaranteed to exist on the domain of stabilizing $\mathbf{K}$ by the generalized Nyquist criterion \cite{RefWorks:176}. 
Conversely, the corresponding controller $\mathbf{K}$ can be recovered from $\mathbf{Q}$ by
$\mathbf{K}=-(I-\mathbf{Q}\mathbf{G}_{22})^{-1}\mathbf{Q}$.
The objective function is expressed as $\|\mathbf{G}_{11}-\mathbf{G}_{12}\mathbf{Q}\mathbf{G}_{21}\|^2$, which is clearly convex with respect to the new parameter $\mathbf{Q}$. 
Moreover, the requirement that $\mathbf{K}$ is stabilizing is translated in the new domain as the requirement that $\mathbf{Q}$ is stable (i.e., $\mathbf{Q}\in H_\infty$)\footnote{This can be intuitively understood via the \emph{internal model principle} \cite{RefWorks:176}.}, which is also a convex constraint.
The information constraint $\mathbf{K}\in\mathcal{I}$, however, results in a nonconvex constraints on the $\mathbf{Q}$ domain, unless $\mathcal{I}$ is \emph{quadratically invariant} under $\mathbf{G}_{22}.$
\begin{definition}(Quadratic Invariance) Let $\mathcal{U},\mathcal{Y}$ be vector spaces. Suppose $\mathbf{G}$ is a linear mapping from $\mathcal{U}$ to $\mathcal{Y}$ and $\mathcal{I}$ is a set of linear maps from $\mathcal{Y}$ to $\mathcal{U}$. Then $\mathcal{I}$ is called quadratically invariant under $\mathbf{G}$ if $\mathbf{K}\in \mathcal{I} \Rightarrow  \mathbf{K}\mathbf{G}\mathbf{K}\in \mathcal{I}$.
\end{definition}
The significance of $\mathcal{I}$ being quadratic invariant is that, the condition $\mathbf{K}\in\mathcal{I}$ can be translated to $\mathbf{Q}\in\mathcal{I}$ on the new domain, which is a convex constraint\footnote{This fact can be understood from the following informal observation: if $\mathbf{K}\in \mathcal{I}$ and $\mathcal{I}$ is QI under $\mathbf{G}_{22}$, the Neumann series expansion gives
$$
\mathbf{Q}=-\mathbf{K}(I-\mathbf{G}_{22}\mathbf{K})^{-1}=-\mathbf{K}-\mathbf{K}\mathbf{G}_{22}\mathbf{K}-\mathbf{K}(\mathbf{G}_{22}\mathbf{K})^2-\cdots.
$$
Since we can inductively show that every term in the last expression belongs to $\mathcal{I}$, $\mathbf{Q}\in \mathcal{I}.$}.

Since we are considering triangular LQG problems in this paper, $\mathbf{G}_{22}$ has a lower block triangular structure ($\mathbf{G}_{22}\in\mathcal{I}_{LBT}$), while the information constraint on the controller implies that the controller transfer function is also lower block triangular ($\mathbf{K}\in\mathcal{I}_{LBT}$). It is easy to verify that the space of transfer function matrices with sparsity pattern $\mathcal{I}_{LBT}$ is indeed quadratically invariant under $\mathbf{G}_{22}$. 
Therefore, Problem \ref{h2controlproblem} can be recast as the following infinite dimensional convex optimization problem.
\begin{subequations}
\label{h2modelmatching1}
\begin{align}
\min \;\; & \|\mathbf{G}_{11}-\mathbf{G}_{12}\mathbf{Q}\mathbf{G}_{21}\| \\
\text{ s.t } \;\; & \mathbf{Q}\in H_\infty \cap \mathcal{I}_{LBT}.
\end{align}
\end{subequations}
An optimization problem of this form is called the \emph{model matching problem}.

If $\mathbf{G}_{22}$ is not stable, Youla parameter must be constructed using coprime factors of $\mathbf{G}_{22}$ \cite{RefWorks:292}.
In this case, parametrization of stabilizing controllers subject to the information constraint is more involved. Nevertheless, if $\mathcal{I}$ is quadratically invariant under $\mathbf{G}_{22}$, the information constraints can be recast as linear constraints on the Youla parameter \cite{RefWorks:277}. Therefore, even in this case, the optimal $H_2$ control problem subject to information constraints can be reformulated as a convex optimization problem.
\fi

\ifdefined\LONGVERSION
\subsection{Centralized $H_2$ Optimal Control}
\label{seccent2}
Assume that $A$ is a Hurwitz matrix and that $\mathbf{G}_{11}\in H_2$. 
In this case, the optimal solution $\mathbf{Q}$ to the $H_2$ model matching problem (\ref{h2modelmatching1}) must be in $H_2$, 
since otherwise $\|\mathbf{G}_{11}-\mathbf{G}_{12}\mathbf{Q} \mathbf{G}_{21}\|$ is unbounded under Assumption \ref{asmp1}. 
Hence the centralized $H_2$ optimal control problem can be cast as a simple $H_2$ model matching problem
\begin{subequations}
\label{h2modelmatching2}
\begin{align}
\min \;\; & \|\mathbf{G}_{11}-\mathbf{G}_{12}\mathbf{Q}\mathbf{G}_{21}\| \\
\text{ s.t } \;\; & \mathbf{Q}\in H_2.
\end{align}
\end{subequations}
One approach to find a solution to the model matching problem is to apply the projection theorem \cite{RefWorks:295}. Let $H$ be a Hilbert space and $S$ be a closed nonempty subspace of $H$. 
If $a\in H$, then there exists an element $x_0\in S$ such that $\|a-x_0\|\leq \|a-x\|$ for all $x\in S$. Such an element $x_0$ is called a projection of $a$ onto $S$, and is denoted by $x_0=\pi_S(a)$.
It can be shown that $x_0=\pi_S(a)$ if and only if $x_0\in S$ and 
$\left<a-x_0,x\right>=0$ for all $x\in S$.

In the model matching problem (\ref{h2modelmatching2}), a set $S\triangleq \{\mathbf{G}_{12}\mathbf{Q}\mathbf{G}_{21}: \mathbf{Q}\in H_2\}$ defines a closed nonempty subspace of $H_2$.
Hence the optimal solution to (\ref{h2modelmatching2}) can be found by computing a projection of $\mathbf{G}_{11}$ onto $S$. 
To this end, a coprime factorization technique of rational functions, particularly the \emph{inner-outer factorization}, is useful. 
We first recall the following facts. Proofs can be found in \cite{RefWorks:21}.
\begin{theorem}
\label{theorceq}
Suppose $(A,B)$ is stabilizable, $H$ has full column rank, and $\left[ \begin{array}{cc} A-j\omega I & B \\ F & H\end{array}\right]$ has full column rank for all $\omega \in \mathbb{R}$. Then algebraic Riccati equation $(X,K)=ARE_p(A,B,F,H)$
% \begin{equation*}
% A^TX+XA+F^TF-(XB+F^TH)\Psi^{-1}(XB+F^TH)^T=0
% \end{equation*}
has a unique positive semidefinite solution. Moreover, it is stabilizing (that is, $A_{cl}\triangleq A+BK$ is a Hurwitz stable matrix). 
\end{theorem}
% \begin{proof}
% It follows from Corollary 13.10 of \cite{RefWorks:21} that there exists a unique stabilizing solution, which is positive semidefinite. Hence, it is enough to prove that any positive semidefinite solution must be stabilizing. To this end, assume that there exists a positive semidefinite solution that is not stabilizing. Rewrite (\ref{theorceq}) as
% $$
% A_{cl}^TX+XA_{cl}+XB\Psi^{-1}B^TX+F^T(I-H\Psi^{-1}H^T)F=0.
% $$
% Let $\lambda$ and $x$ be an unstable eigenvalue and eigenvector of $A_{cl}$. Multiplied by $x^*$ from left and $x$ from right, we have
% $$
% (\lambda+\bar{\lambda})x^*Xx+x^*XB\Psi^{-1}B^TXx+x^*F^T(I-H\Psi^{-1}H^T)Fx=0.
% $$
% Since every term on the left hand side is nonnegative, we need $B^TXx=0$ and $(I-H\Psi^{-1}H^T)Fx=0$. Thus
% $$
% A_{cl}x=(A-B\Psi^{-1}H^TF)x=j\omega \text{ and } (I-H\Psi^{-1}H^T)Fx=0.
% $$
% This implies that $j\omega$ is an unobservable mode of $((I-H\Psi^{-1}H^T)F, A-B\Psi^{-1}H^TF)$. By Lemma 13.9 of \cite{RefWorks:21}, this contradicts the assumption that $\left[ \begin{array}{cc} A-j\omega I & B \\ F & H\end{array}\right]$ has full column rank for all $\omega in \mathbb{R}$. 
% \end{proof}
\begin{theorem}
\label{theoiof}
\begin{itemize}
\item[(1).] (Right Inner-Outer Factorization) 

Assume $\mathbf{G}_{12}=\left[ \begin{array}{c|c} A&B\\ \hline  F&H\end{array}\right]$ is stabilizable and $\left[ \begin{array}{cc} A-j\omega I & B \\ F & H \end{array}\right]$ has full column rank for all $\omega\in \mathbb{R}$. 
Then there exists a right coprime factorization $\mathbf{G}_{12}=\mathbf{U}\mathbf{M}^{-1}$ such that $\mathbf{U}$ is inner and $\mathbf{M}$ is stably invertible. A particular realization of such factorization is
$$
\mathbf{U}=\left[ \begin{array}{c|c}
A+BK & B\Psi^{-\frac{1}{2}} \\ \hline 
F+HK & H\Psi^{-\frac{1}{2}} \end{array}\right], \;
\mathbf{M}^{-1}=\left[ \begin{array}{c|c}
A&B \\ \hline  
-\Psi^{\frac{1}{2}}K & \Psi^{\frac{1}{2}} \end{array}\right]
$$
where $\Psi=H^TH$ and $(X,K)=ARE_p(A,B,F,H)$.
\item[(2).] (Left Inner-Outer Factorization) 

Assume $\mathbf{G}_{21}=\left[ \begin{array}{c|c}A & W\\ \hline 
C&V\end{array}\right]$ is detectable and $\left[ \begin{array}{cc}A-j\omega I & W\\ 
C&V\end{array}\right]$ has full row rank for all $\omega\in \mathbb{R}$. Then there exists a left coprime factorization $\mathbf{G}_{21}=\mathbf{N}^{-1}\mathbf{V}$ such that $\mathbf{V}$ is co-inner and $\mathbf{N}$ is stably invertible. 
A particular realization of such factorization is 
$$
\mathbf{V}=\left[ \begin{array}{c|c} 
A+LC & W+LV\\ \hline 
\Phi^{-\frac{1}{2}}C &\Phi^{-\frac{1}{2}}V  \end{array}\right], \;
\mathbf{N}^{-1}=\left[ \begin{array}{c|c} 
A&-L\Phi^{\frac{1}{2}}\\ \hline 
C&\Phi^{\frac{1}{2}}\end{array}\right] 
$$
where $\Phi=VV^T$ and $(Y,L)=ARE_d(A,C,W,V)$.
\end{itemize}
\end{theorem}
Using $\mathbf{U}$ and $\mathbf{V}$, the subspace $S$ can be expresses as $
S=\left\{\mathbf{U}\mathbf{P}\mathbf{V}: \mathbf{P}\in H_2\right\}$, where $\mathbf{P}=\mathbf{M}^{-1}\mathbf{Q}\mathbf{N}^{-1}$ is a new parameter. Since $\mathbf{M}$ and $\mathbf{N}$ are stably invertible, requiring 
that $\mathbf{Q}\in H_2$ is equivalent to requiring $\mathbf{P}\in H_2$. 

Suppose the projection of $\mathbf{G}_{11}$ onto $S$ can be written as $\mathbf{U}\mathbf{P}'\mathbf{V}$. 
By the optimality condition $\left<\mathbf{G}_{11}-\mathbf{U}\mathbf{P}'\mathbf{V}, \mathbf{U}\mathbf{P}\mathbf{V}\right>=0$ $\forall \mathbf{P}\in H_2$, one obtains $\mathbf{U}^*\mathbf{G}_{11}\mathbf{V}^*-\mathbf{P}'\in H_2^\perp$. Hence it must be that
$\mathbf{P}'=\pi_{H_2}(\mathbf{U}^*\mathbf{G}_{11}\mathbf{V}^*)$. Moreover, straightforward state space manipulations show that
\begin{equation}
\label{p_cf}
\mathbf{P}'=\pi_{H_2}(\mathbf{U}^*\mathbf{G}_{11}\mathbf{V}^*)=
\left[ \begin{array}{c|c} 
A&-L\Phi^{\frac{1}{2}} \\ \hline 
-\Psi^{\frac{1}{2}}K& 0 
\end{array}\right].
\end{equation}
From this result, one can recover the optimal centralized $H_2$ controller.
$$
\mathbf{K}_{opt}=\left[ \begin{array}{c|c}
A+BK+LC & -L \\ \hline 
K & 0 \end{array}\right].
$$
It is also possible to show that the optimal control performance is
\begin{equation}
\label{jcent}
J_{cent}=tr(W^TXW)+tr(\Psi KYK^T).
\end{equation}
\fi

\subsection{Proof of Proposition \ref{propstabsol}} 
By Theorem \ref{theorceq}, $(\hat{X}_1,\hat{K}_1)=ARE_p(A,B,F,H)$ has a unique positive semidefinite solution that is also stabilizing; in particular $A_{\downarrow 1}^{\downarrow 1}+B_{\downarrow 1}^{\downarrow 1}\hat{K}_1$ is stable. 
Thus it suffices to show that if $A_{\downarrow i}^{\downarrow i}+B_{\downarrow i}^{\downarrow i}\hat{K}_i$ is stable for some $i\in\{1,2,\cdots,N-1\}$, then the algebraic Riccati equation
\begin{equation}
\label{rciplus1}
(\hat{X}_{i+1},\hat{K}_{i+1})=ARE_p(A_{\downarrow i+1}^{\downarrow i+1},B_{\downarrow i+1}^{\downarrow i+1},-H^{\downarrow i}\hat{K}_i^b,H^{\downarrow i+1})
\end{equation}
has a unique positive semidefinite solution that is also stabilizing. 
Since $(A_{\downarrow i+1}^{\downarrow i+1},B_{\downarrow i+1}^{\downarrow i+1})$ is stabilizable by Assumption \ref{asmp1}, by Theorem \ref{theorceq}, the only way that (\ref{rciplus1}) fails to have a unique positive semidefinite solution is that 
$
\left[\! \begin{array}{cc}
A_{\downarrow i+1}^{\downarrow i+1}-j\omega I & B_{\downarrow i+1}^{\downarrow i+1} \\
-H^{\downarrow i}\hat{K}_i^b & H^{\downarrow i+1}
\end{array}\!\right]\!
\left[\! \begin{array}{c} z_1 \\ z_2 \end{array}\!\right]\!=\!0
\text{ for some }
\left[\! \begin{array}{c} z_1 \\ z_2 \end{array}\!\right]\!\neq\! 0.
$
Notice that $z_1$ cannot be zero as this implies $z_2$ is also zero due to the assumption that $H^{\downarrow i+1}$ has full column rank. 
Introducing
$
\tilde{z}_1=\left[ \begin{array}{c} 0 \\ z_1 \end{array}\right], \;
\tilde{z}_2=\left[ \begin{array}{c} 0 \\ z_2 \end{array}\right],
$
we have that
$
\left[ \begin{array}{cc}
A_{\downarrow i}^{\downarrow i}-j\omega I & B_{\downarrow i}^{\downarrow i} \\
-H^{\downarrow i}\hat{K}_i & H^{\downarrow i}
\end{array}\right]
\left[ \begin{array}{c} \tilde{z}_1 \\ \tilde{z}_2 \end{array}\right]=0.
$
Left-multiplied by $\left[ \begin{array}{cc} I & -B_{\downarrow i}^{\downarrow i}(\Psi_{\downarrow i}^{\downarrow i})^{-1} {H^{\downarrow i}}^T  \end{array}\right]$, this reduces to
$
j\omega \tilde{z}_1=(A_{\downarrow i}^{\downarrow i}+B_{\downarrow i}^{\downarrow i}\hat{K}_i)\tilde{z}_1, \;\; \tilde{z}_1\neq 0.
$
However, this contradicts that stability of $A_{\downarrow i}^{\downarrow i}+B_{\downarrow i}^{\downarrow i}\hat{K}_i$. This shows that (\ref{rcicsub}) has a unique positive semidefinite and stabilizing solution for all $i\in \{2,\cdots,N\}$.
A similar argument can be applied to (\ref{rcifsub}) as well.

\subsection{Proof of Lemma \ref{lempij}} \label{proof:lempij}
\begin{lemma}\label{lemUM}
(1). For every $i\in\{1,2,\cdots,N\}$, $\mathbf{U}_i$ is an inner function and $\mathbf{G}_{12}^{cl}E^{\downarrow i}=\mathbf{U}_1\mathbf{U}_2\cdots\mathbf{U}_i\mathbf{M}_i^{-1}$. 

(2). For every $j\in\{1,2,\cdots,N\}$, $\mathbf{V}_j$ is a co-inner function and $E_{\uparrow j}\mathbf{G}_{21}^{cl}=\mathbf{N}_j^{-1}\mathbf{V}_j\mathbf{V}_{j+1}\cdots\mathbf{V}_N$.
\end{lemma}
\begin{proof}
(1). When $i=1$,
\begin{align*}
\mathbf{U}_1\mathbf{M}_1^{-1}&=\left[ \begin{array}{cc|c}
A_{1,0}^{KL} & -B\mathcal{K}_1 & B \\
0 & \mathcal{A} & \mathcal{B}  \\ \hline
F+HK_1 & -H\mathcal{K}_1 & H
\end{array}\right] \\
&=\left[ \begin{array}{cc|c}
A_{1,0}^{KL} & 0 & 0\\
0 & \mathcal{A} & \mathcal{B} \\ \hline
F+HK_1 & \mathcal{F} & H
\end{array}\right] 
=\mathbf{G}_{12}^{cl}.
\end{align*}
From the first to the second expression, a similarity transformation was applied by multiplying the state space ``A" matrix by
$\left[ \begin{array}{cc} I & \tilde{J}_1 \\ 0 & I \end{array}\right]$ from the left and by its inverse from the right. The fact that $K_1$ is a solution to (\ref{rc1c}) is exploited to obtain ``0'' at the upper-right corner of the state space ``A'' matrix. The last step was obtained by eliminating uncontrollable states.
When $i\in\{2,3,\cdots,N\}$,
\begin{align*}
\mathbf{U}_i\mathbf{M}_i^{-1}&\!\!\!=\!\!\!\left[\! \begin{array}{cc|c}
A_{i,i-1}^{KL} \!\!&\!\! -B^{\downarrow i}\mathcal{K}_i \!\!&\!\! B^{\downarrow i} \\
0 \!\!&\!\! \mathcal{A} \!\!&\!\! \mathcal{B}^{\downarrow i}  \\ \hline
C_{U_i} \!\!\!&\!\!\! -{\Psi_{\downarrow i-1}^{\downarrow i-1}}^\frac{1}{2}E_{\downarrow i-1}E^{\downarrow i}\mathcal{K}_i \!\!&\!\!  {\Psi_{\downarrow i-1}^{\downarrow i-1}}^\frac{1}{2}E_{\downarrow i-1}E^{\downarrow i}\!\!
\end{array}\!\right] \\
&\!\!=\!\!\!\left[ \begin{array}{cc|c}
A_{i,i-1}^{KL} & 0 & 0\\
0 & \mathcal{A} & \mathcal{B}^{\downarrow i}  \\ \hline
C_{U_i} & -{\Psi_{\downarrow i-1}^{\downarrow i-1}}^\frac{1}{2}\mathcal{K}_{i-1} & {\Psi_{\downarrow i-1}^{\downarrow i-1}}^\frac{1}{2}E_{\downarrow i-1}E^{\downarrow i}
\end{array}\right] \\
&\!\!=\! \mathbf{M}_{i-1}^{-1}E_{\downarrow i-1}E^{\downarrow i}.
\end{align*}
From the first to the second expression, a similarity transformation was applied by multiplying by
$\left[ \begin{array}{cc} I & \tilde{J}_i \\ 0 & I \end{array}\right]$ from the left and by its inverse from the right. Combining the above results, one obtains $\mathbf{G}_{12}^{cl}E^{\downarrow i}=\mathbf{U}_1\mathbf{U}_2\cdots\mathbf{U}_i\mathbf{M}_i^{-1}$.
To see that $\mathbf{U}_i$ is inner, notice that
\begin{align*}
{A_{i,i-1}^{KL}}^TX_i+X_iA_{i,i-1}^{KL}+C_{U_i}^TC_{U_i}&=0 \\
D_{U_i}^TC_{U_i}+B_{U_i}^TX_i&=0
\end{align*}
follows from (\ref{rcicr}). Since $D_{U_i}^TD_{U_i}=I$, by Corollary 13.30 in \cite{RefWorks:21}, $\mathbf{U}_i$ is an inner function.

(2). When $j=N$,
\begin{align*}
\mathbf{N}_N^{-1}\mathbf{V}_N&=\left[ \begin{array}{cc|c}
\mathcal{A} & -\mathcal{L}_NC & -\mathcal{L}_NV \\
0 & A_{N+1,N}^{KL} & W+L_NV  \\ \hline
\mathcal{C} & C & V
\end{array}\right] \\ 
&=\left[ \begin{array}{cc|c}
\mathcal{A} & 0 & \mathcal{W} \\
0 & A_{N+1,N}^{KL} & W+L_NV \\ \hline
\mathcal{C} & 0 & V
\end{array}\right] 
=\mathbf{G}_{21}^{cl}.
\end{align*}
From the first to the second expression, a similarity transformation was applied by multiplying the state space ``A" matrix by
$\left[ \begin{array}{cc} I & \hat{J}_N \\ 0 & I \end{array}\right]$ from the left and by its inverse from the right. The last step was obtained by eliminating unobservable states.
When $j\in\{1,2,\cdots,N-1\}$,
\begin{align*}
\mathbf{N}_j^{-1}\mathbf{V}_j&=\left[ \begin{array}{cc|c}
\mathcal{A} & -\mathcal{L}_jC_{\uparrow j} & -\mathcal{L}_jE_{\uparrow j}E^{\uparrow j+1}{\Phi_{\uparrow j+1}^{\uparrow j+1}}^{\frac{1}{2}} \\
0 & A_{j+1,j}^{KL} & B_{V_j}  \\ \hline
\mathcal{C}_{\uparrow j} & C_{\uparrow j} &  E_{\uparrow j}E^{\uparrow j+1}{\Phi_{\uparrow j+1}^{\uparrow j+1}}^\frac{1}{2}
\end{array}\right] \\
&=\left[ \begin{array}{cc|c}
\mathcal{A} & 0 & -\mathcal{L}_{j+1}{\Phi_{\uparrow j+1}^{\uparrow j+1}}^{\frac{1}{2}} \\
0 & A_{j+1,j}^{KL} & B_{V_j}  \\ \hline
\mathcal{C}_{\uparrow j} & 0 &  E_{\uparrow j}E^{\uparrow j+1}{\Phi_{\uparrow j+1}^{\uparrow j+1}}^\frac{1}{2}
\end{array}\right]\\
&=E_{\uparrow j}E^{\uparrow j+1}\mathbf{N}_{i+1}^{-1}.
\end{align*}
From the first to the second expression, a similarity transformation was applied by multiplying by
$\left[ \begin{array}{cc} I & \hat{J}_j \\ 0 & I \end{array}\right]$ from the left and by its inverse from the right. 
Combining the above results, one obtains $E_{\uparrow j}\mathbf{G}_{21}^{cl}=\mathbf{N}_j^{-1}\mathbf{V}_j\mathbf{V}_{j+1}\cdots\mathbf{V}_N$.
To see that $\mathbf{V}_j$ is co-inner, or equivalently that $\mathbf{V}_j^T$ is inner, notice that 
\begin{align*}
A_{j+1,j}^{KL}Y_j+Y_j{A_{j+1,j}^{KL}}^T+B_{V_j}B_{V_j}^T&=0 \\
D_{V_j}B_{V_j}^T+C_{V_j}Y_j&=0
\end{align*}
follows from (\ref{rcifr}). Since $D_{V_j}D_{V_j}^T=I$, $\mathbf{V}_j$ is a co-inner function.
\end{proof}
\begin{lemma}
\label{lemproj}
Let $\Gamma_i$ and $\Lambda_j$ be given as in Lemma \ref{lempij}, while $\Gamma$ and $\Lambda$ be any matrices. Let $0\leq i < j\leq N+1$.
\begin{itemize}
\item[(1).] If $\mathbf{F}=\mathbf{U}_1\cdots \mathbf{U}_i\tilde{\mathbf{P}}_i \mathbf{V}_j\cdots \mathbf{V}_N \in S_{i,j}$ where
$$
\tilde{\mathbf{P}}_i \triangleq \left[ \begin{array}{c|c}
\mathcal{A} & \Lambda \\ \hline 
 \Gamma_i & 0
\end{array}\right], 
$$
then $\pi_{S_{i+1,j}}(\mathbf{F})=\mathbf{U}_1\cdots \mathbf{U}_{i+1}\tilde{\mathbf{P}}_{i+1} \mathbf{V}_j\cdots \mathbf{V}_N$.
\item[(2).] If $\mathbf{F}=\mathbf{U}_1\cdots \mathbf{U}_i\hat{\mathbf{P}}_j \mathbf{V}_j\cdots \mathbf{V}_N \in S_{i,j}$ where
$$
\hat{\mathbf{P}}_j \triangleq \left[ \begin{array}{c|c}
\mathcal{A} & \Lambda_j \\ \hline 
 \Gamma & 0
\end{array}\right],
$$
then $\pi_{S_{i,j-1}}(\mathbf{F})=\mathbf{U}_1\cdots \mathbf{U}_{i}\hat{\mathbf{P}}_{i-1} \mathbf{V}_{j-1}\cdots \mathbf{V}_N$.
\end{itemize}
\end{lemma}
\begin{proof}
(1). We show that if $\tilde{\mathbf{P}}_i$ is in the assumed form, then so is $\tilde{\mathbf{P}}_{i+1}$. By the optimality condition, $\pi_{S_{i+1,j}}(\mathbf{F})$ satisfies
$$
\left<\mathbf{F}\!-\!\pi_{S_{i+1,j}}\!(\mathbf{F}), \mathbf{U}_1\!\cdots\!\mathbf{U}_{i+1}\mathbf{M}^{-1}_{i+1}\mathbf{Q}_{i+1,j}\mathbf{N}^{-1}_{j}\mathbf{V}_j\!\cdots\!\!\mathbf{V}_N \right>\!\!=\!0  
$$
for all $\mathbf{Q}_{i+1,j}\in H_2$. 
Hence $\tilde{\mathbf{P}}_{i+1}$ must satisfy
$$
\left<\mathbf{U}_{i+1}^*\tilde{\mathbf{P}}_{i}-\tilde{\mathbf{P}}_{i+1}, \mathbf{M}_{i+1}^{-1}\mathbf{Q}_{i+1,j}\mathbf{N}_{j}^{-1}\right>=0 \; \forall \mathbf{Q}_{i+1,j}\in H_2.
$$ 
Such $\tilde{\mathbf{P}}_{i+1}$ is given by $\tilde{\mathbf{P}}_{i+1}=\pi_{H_2}(\mathbf{U}_{i+1}^*\tilde{\mathbf{P}}_i)$, since in this case, the inner product of $\mathbf{U}_{i+1}^*\tilde{\mathbf{P}}_i-\tilde{\mathbf{P}}_{i+1}\in H_2^\perp$ and $\mathbf{M}_{i+1}^{-1}\mathbf{Q}_{i+1,j}\mathbf{N}_{j}^{-1}\in H_2$ is zero.
The projection $\pi_{H_2}(\mathbf{U}_{i+1}^*\tilde{\mathbf{P}}_i)$ is computed as follows. For every $i\in\{0,1,\cdots, N-1\}$,
\begin{align*}
\mathbf{U}_{i+1}^*\tilde{\mathbf{P}}_i
&=\left[ \begin{array}{cc|c}
-{A_{i+1,i}^{KL}}^T & C_{U_{i+1}}^T\Gamma_i & 0 \\
0 & \mathcal{A} & \Lambda \\ \hline
-B_{U_{i+1}}^T & D_{U_{i+1}}^T\Gamma_i & 0\end{array}\right] \\
&=\left[ \begin{array}{cc|c}
-{A_{i+1,i}^{KL}}^T & 0 & * \\
0 & \mathcal{A} & \Lambda \\ \hline
*& \Gamma_{i+1} & 0\end{array}\right]
\end{align*}
To obtain the last expression, a similarity transformation is applied to the state space matrices by left-multiplying by $\left[ \begin{array}{cc} I & -X_{i+1}\tilde{J}_{i+1} \\ 0 & I\end{array}\right]$ and right-multiplying by its inverse. 
Notice that the Riccati equation (\ref{rc1c}) or (\ref{rcic}) appears on the upper right block of the state space ``A" matrix, and hence this component is zero. 
Moreover, since $\mathcal{A}$ is stable and $-{A_{1,0}^{KL}}^T$ is anti-stable, its projection onto $H_2$ is given by
$
\tilde{\mathbf{P}}_{i+1}=\pi_{H_2}(\mathbf{U}_{i+1}^*\tilde{\mathbf{P}}_i)=\left[ \begin{array}{c|c}\mathcal{A} & \Lambda \\ \hline \Gamma_{i+1} & 0 \end{array}\right].
$

(2). Similarly, we show that if $\hat{\mathbf{P}}_{j}$ is in the assumed form, then so is $\hat{\mathbf{P}}_{j-1}$.
From the optimality condition, it is possible to infer that $\hat{\mathbf{P}}_{j-1}=\pi_{H_2}(\hat{\mathbf{P}}_{j}\mathbf{V}_{j-1}^*)$. For every $j\in\{1,2,\cdots, N+1\}$, 
\begin{align*}
\hat{\mathbf{P}}_{j}\mathbf{V}_{j-1}^*&=\left[ \begin{array}{cc|c}
\mathcal{A} & \Lambda_j B_{V_{j-1}}^T & \Lambda_j D_{V_{j-1}}^T \\
0 & -{A_{j,j-1}^{KL}}^T & -C_{V_{j-1}}^T \\ \hline
\Gamma & 0 & 0
 \end{array}\right] \\
&=\left[ \begin{array}{cc|c}
\mathcal{A} & 0 & \Lambda_{j-1} \\
0 & -{A_{j,j-1}^{KL}}^T & * \\ \hline
\Gamma & * & 0
 \end{array}\right].
\end{align*}
A similarity transformation is applied to the state space matrices by left-multiplying by $\left[ \begin{array}{cc} I & -\hat{J}_{j-1}Y_{j-1} \\ 0 & I\end{array}\right]$ and right-multiplying by its inverse. Since $-{A_{j,j-1}^{KL}}^T$ is anti-stable, its projection onto $H_2$ is 
$
\hat{\mathbf{P}}_{j-1}=\pi_{H_2}(\hat{\mathbf{P}}_{j}\mathbf{V}_{j-1}^*)=\left[ \begin{array}{c|c}\mathcal{A} & \Lambda_{j-1} \\ \hline \Gamma & 0 \end{array}\right].
$
\end{proof}

Proof of Lemma \ref{lempij} is by induction. When $i=0$ and $j=N+1$, the identity (\ref{p_ij}) 
clearly holds since $\mathbf{P}'_{0,N+1}=\mathbf{G}_{11}^{cl}$.
So suppose (\ref{p_ij}) holds for some $(i,j)$ such that $0\leq i<j\leq N+1$.
By the nested projection, 
$\pi_{S_{i+1,j}}(\mathbf{G}_{11}^{cl})=\pi_{S_{i+1,j}}(\mathbf{F})$, where $\mathbf{F}=\pi_{S_{i,j}}(\mathbf{G}_{11}^{cl})=\mathbf{U}_1\cdots \mathbf{U}_i \mathbf{P}_{i,j}\mathbf{V}_j\cdots \mathbf{V}_N$.
Applying Lemma \ref{lemproj} (1), we have that 
$\pi_{S_{i+1,j}}(\mathbf{G}_{11}^{cl})=\mathbf{U}_1\cdots \mathbf{U}_{i+1} \mathbf{P}_{i+1,j}\mathbf{V}_j\cdots \mathbf{V}_N$.
Hence, we have verified that (\ref{p_ij}) holds for $(i+1,j)$.
Similarly, by the nested projection,
$\pi_{S_{i,j-1}}(\mathbf{G}_{11}^{cl})=\pi_{S_{i,j-1}}(\mathbf{F})$.
Applying Lemma \ref{lemproj} (2), we have that 
$\pi_{S_{i,j-1}}(\mathbf{G}_{11}^{cl})=\mathbf{U}_1\cdots \mathbf{U}_{i} \mathbf{P}_{i,j-1}\mathbf{V}_{j-1}\cdots \mathbf{V}_N$.
Hence, we have verified that (\ref{p_ij}) holds for $(i,j-1)$.
This proves that the identity (\ref{p_ij}) holds for every subspace in Fig. \ref{figdiagram}.
Finally, state space expressions for $\mathbf{R}_{(i-1,j)\rightarrow (i,j)}=\mathbf{P}'_{i-1,j}-\mathbf{U}_i\mathbf{P}'_{i,j}$ and $\mathbf{R}_{(i,j+1)\rightarrow (i,j)}=\mathbf{P}'_{i,j+1}-\mathbf{P}'_{i,j}\mathbf{V}_j$ are obtained by  straightforward state space manipulations.

\subsection{Certainty Equivalence}
\label{secce}
We verify that $x^{K_i}(t)$ can be interpreted as the least mean square estimate of $x(t)$ conditioned on the observations of outputs of upstream subsystems.
In other words, if $\mathbf{H}_{y_j}:w\mapsto y_{\uparrow j}$ and $\mathbf{F}:w\mapsto x$ are given, a transfer function $\mathbf{Q}_{x_j}:y_{\uparrow j}\mapsto x^{K_j}$ defined by the proposed controller minimizes $\|\mathbf{F}-\mathbf{Q}_{x_j}\mathbf{H}_{y_j}\|$ over $H_2$.
Notice that $\mathbf{H}_{y_j}=E_{\uparrow j}\mathbf{G}_{21}^{cl}$,
$$
\mathbf{F}=\left[ \begin{array}{c|c} 
\mathcal{A} & \mathcal{W} \\ \hline 
\text{row}\{0,\cdots,0,I\} & 0 \end{array}\right],
\mathbf{Q}_{x_j}=\left[ \begin{array}{c|c} 
A_K & B_KE^{\uparrow j} \\ \hline 
E_j & 0 \end{array}\right].
$$
By the optimality condition, it suffices to check that $\pi_{S_{0,j}}(\mathbf{F})=\mathbf{Q}_{x_j}\mathbf{H}_{y_j}$.
Applying Lemma \ref{lemproj} (2) repeatedly, the LHS becomes $\pi_{S_{0,j}}(\mathbf{F})=\hat{\mathbf{P}}_j\mathbf{V}_j\cdots\mathbf{V}_N$ where
$\hat{\mathbf{P}}_j=\left[ \begin{array}{c|c} 
\mathcal{A} & -\mathcal{L}_j{\Phi_{\uparrow j}^{\uparrow j}}^{\frac{1}{2}} \\ \hline 
\text{row}\{0,\cdots,0,I\} & 0 \end{array}\right]$,
while the RHS is
\begin{align*}
\mathbf{Q}_{x_j}\mathbf{H}_{y_j}&=\mathbf{Q}_{x_j}\mathbf{N}_j^{-1}\mathbf{V}_j\cdots\mathbf{V}_N \\
&=\left[ \begin{array}{c|c} 
\mathcal{A} & -\mathcal{L}_j{\Phi_{\uparrow j}^{\uparrow j}}^{\frac{1}{2}} \\ \hline 
\text{row}\{ E_j,0\} & 0 \end{array}\right]\mathbf{V}_j\cdots \mathbf{V}_N.
\end{align*}
To see $\hat{\mathbf{P}}_j=\mathbf{Q}_{x_j}\mathbf{H}_{y_j}$, notice that
\begin{align}
\hat{\mathbf{P}}_j-\mathbf{Q}_{x_j}\mathbf{H}_{y_j}&=
\left[ \begin{array}{c|c} 
\mathcal{A} & -\mathcal{L}_j{\Phi_{\uparrow j}^{\uparrow j}}^{\frac{1}{2}} \\ \hline 
\text{row}\{-E_j,I\} & 0 \end{array}\right] \nonumber  \\ 
&=
\left[ \begin{array}{c|c} 
\bar{\mu}\mathcal{A}\bar{\zeta} & -\bar{\mu}\mathcal{L}_j{\Phi_{\uparrow j}^{\uparrow j}}^{\frac{1}{2}} \\ \hline 
\text{row}\{-E_j,I\}\bar{\zeta} & 0 \end{array}\right].  \label{eqzero}
\end{align}
The upper right block of (\ref{eqzero}) has nonzero matrices only on the first $j$ subblocks, while the first $j$ subblocks of the lower left block of (\ref{eqzero}) are all zero matrices. Since $\bar{\mu}\mathcal{A}\bar{\zeta}$ is upper-triangular as observed in (\ref{muAzeta}), all controllable states are not observable and (\ref{eqzero}) is identically zero.
This proves $\pi_{S_{0,j}}(\mathbf{F})=\mathbf{Q}_{x_j}\mathbf{H}_{y_j}$.

%\subsection{Proof of Lemma \label{lemproj}}

%% References with bibTeX database:

\end{document}